\documentclass[11pt,a4paper]{article}

\usepackage{graphicx}
\usepackage{amsmath}
\usepackage{amsfonts}
\usepackage{amssymb}
\usepackage{enumerate}
\usepackage{lscape}
\usepackage{longtable}
\usepackage{rotating}
\usepackage{color}
\usepackage{url}
\usepackage[font={small,it}]{caption}
\usepackage{subcaption}
\usepackage{rotating}
\usepackage{titlesec}
\usepackage[title]{appendix}
\newcommand{\norm}[1]{\left\lVert#1\right\rVert}
\usepackage{multicol}
\usepackage{mathrsfs}
\usepackage{upgreek}
\usepackage{lmodern}
\usepackage{algorithmicx}
\usepackage{algorithm}
\usepackage[noend]{algpseudocode}

\newtheorem{theorem}{Theorem}[section]

\newtheorem{lemma}{Lemma}[section]

\newtheorem{proposition}{Proposition}[section]
\newtheorem{remark}{Remark}[section]

\newtheorem{assumption}{Assumption}[section]

\newenvironment{proof}[1][Proof]{\textbf{#1.} }{\hfill$\Box$}

\makeatletter
\newcommand*{\centerfloat}{%
  \parindent \z@
  \leftskip \z@ \@plus 1fil \@minus \textwidth
  \rightskip\leftskip
  \parfillskip \z@skip}
\makeatother

\newcommand{\ab}[1]{{\color{blue}#1}}

\usepackage{soul}
\newcommand{\FLE}{\texttt{Full-Low Evaluation}}
\newcommand{\DFOTR}{\texttt{DFO-TR}}

\newcommand{\FullEval}{\texttt{Full-Eval}}
\newcommand{\LowEval}{\texttt{Low-Eval}}

\newcommand{\FDLM}{\texttt{FDLM}}
\newcommand{\NOMAD}{\texttt{NOMAD}}

\newcommand\cuter{{\sf CUTEst}}

\numberwithin{equation}{section}
\definecolor{forestgreen}{rgb}{0.13, 0.55, 0.13}
\definecolor{greenp}{rgb}{0.0, 0.65, 0.31}
\definecolor{greenr}{rgb}{0.4, 0.69, 0.2}
\definecolor{indiagreen}{rgb}{0.07, 0.53, 0.03}

\usepackage[square,numbers]{natbib}
\usepackage{hyperref}
\hypersetup{
  colorlinks,
  citecolor=indiagreen,
  linkcolor=indiagreen,
  urlcolor=Blue}

\begin{document}

\title{Full-low evaluation methods for derivative-free optimization}

\author{
A. S. Berahas\thanks{Department of Industrial \& Operations Engineering,
University of Michigan,
 1205 Beal Avenue, Ann Arbor, MI 48109-2102, USA
({\tt albertberahas@gmail.com}).}
\and
O. Sohab\thanks{Department of Industrial and Systems Engineering,
Lehigh University,
200 West Packer Avenue, Bethlehem, PA 18015-1582, USA
({\tt ous219@lehigh.edu}).}
\and
L. N. Vicente\thanks{Department of Industrial and Systems Engineering,
Lehigh University,
200 West Packer Avenue, Bethlehem, PA 18015-1582, USA. Support for this
author was partially provided by the Centre for Mathematics
of the University of Coimbra under grant FCT/MCTES
UIDB/MAT/00324/2020.}
}

\maketitle
\footnotesep=0.4cm
{\small
\begin{abstract}
We propose a new class of rigorous methods for derivative-free optimization with the aim of delivering efficient and robust numerical performance for functions of all types, from smooth to non-smooth, and under different noise regimes. To this end, we have developed a class
of methods, called \FLE{} methods, organized around two main types of iterations.

The first iteration type (called \FullEval{}) is expensive in function evaluations, but exhibits good performance in the smooth and non-noisy cases. For the theory, we consider a line search based on an approximate gradient, backtracking until a sufficient decrease condition is satisfied. In practice, the gradient was approximated via finite differences, and the direction was calculated by a quasi-Newton step (BFGS). The second iteration type (called \LowEval{}) is cheap in function evaluations, yet more robust in the presence of noise or non-smoothness. For the theory, we consider direct search, and in practice we use probabilistic direct search with one random direction and its negative.

A switch condition from \FullEval{} to \LowEval{} iterations was developed based on the values of the line-search and direct-search stepsizes. If enough \FullEval{} steps are taken, we derive a complexity result of gradient-descent type. Under failure of \FullEval{}, the \LowEval{} iterations become the drivers of convergence yielding non-smooth convergence. \FLE{} methods are shown to be efficient and robust in practice across problems with different levels of smoothness and noise.
\end{abstract}
}

\section{Introduction}
\label{sec:intro}

Derivative-Free Optimization (DFO) methods~\cite{CAudet_WHare_2017,ARConn_KScheinberg_LNVicente_2009b,ALCustodio_KScheinberg_LNVicente_2017,JLarson_MMenickelly_SWild_2019,LMRios_NVSahinidis_2013} are developed for the minimization of functions whose corresponding derivatives are unavailable for use or expensive to compute or approximate. The value of the functions is often computed via numerical simulations and may be subject to statistical noise or other forms of inaccuracy. Constraints may be part of the problem formulation and their derivatives may also be unavailable. DFO methods have applications in all fields of engineering and applied science (see the recent survey~\cite{SAlarie_etal_2022}), in particular whenever data is fitted with the purpose of finding optimal values for design or control variables of complex models or systems, or when there is the need to determine optimal parameters of computational solvers or to tune hyperparameters in artificial intelligence.

In this paper, we consider the following unconstrained problem
\begin{align}
    \min_{x \in \mathbb{R}^n} \; f(x),
\end{align}
where $f: \mathbb{R}^n \to \mathbb{R}$ is assumed to have different properties throughout the paper.
We will consider both the cases where~$f$ is smooth (assuming that it has a Lipschitz continuous gradient) and is possibly non-smooth (assuming only local Lipschitz continuity of~$f$).
The methods discussed and developed in this paper find application on problems where derivative information of whatever form is not available or is too expensive to compute\ab{~\cite{SAlarie_etal_2022}}.
Although we do not consider the noisy case explicitly in our theory, we do take it into account when developing and numerically testing the methods.

\subsection{Advantages and disadvantages of the current DFO methods} \label{sec:adv}

When designing numerical algorithms for DFO~\cite{ARConn_KScheinberg_LNVicente_2009b}, two main algorithmic paradigms are often used, in part depending on how~$f$ is sampled, leading to algorithms which are either \textit{directional} or \textit{model-based}.

Directional algorithms are based on the concept of a displacement along a direction. Among those, we find directional direct-search methods~\cite{CAudet_JEDennis_2006,AJBooker_etal_1998,SGratton_et_al_2015,TGKolda_RMLewis_VTorczon_2003,VTorczon_1997}, which rely only on function evaluations without any implicit or explicit approximation of the gradient, or construction of a model.
At each iteration, a finite set of directions is first generated, from which a set of polling points is considered by adding to the current iterate these
directions multiplied by a certain stepsize. The objective function is then evaluated at all, or some, of these points depending on the polling type (opportunistic or complete), in order to search for a new point satisfying a (possibly sufficient) decrease condition. The iterate and stepsize are updated according to the outcome of the polling step. A search step~\cite{AJBooker_etal_1998} can be taken to improve numerical performance, with no influence on the convergence analysis.
A relevant feature of these methods is their ability to converge when the function is non-smooth. In fact, if the normalized directions are asymptotically dense in the unit sphere, the iterates converge to a Clarke stationary point~\cite{CAudet_JEDennis_2006,LNVicente_ALCustodio_2012}\footnote{Density of search directions is not the single technique that leads to convergence results in non-smooth~DFO. Other major approaches include the use of convex hulls of (possibly randomly) sampled approximate gradients~\cite{AMBagirov_BKarasozen_MSezer_2006,WHare_JNutini_2013,KCKiwiel_2010} or trust-region models~\cite{GLiuzzi_etal_2019} with randomized linear terms.}.

The directional direct-search framework has two main variants, a deterministic one~\cite{TGKolda_RMLewis_VTorczon_2003,VTorczon_1997} and a probabilistic one~\cite{SGratton_et_al_2015}. In the deterministic variant, the directions belong to positive spanning sets (PSS), which are sets of vectors that span the whole space with non-negative coefficients. When the objective function is smooth, at least one of these directions is a descent one~\cite{TGKolda_RMLewis_VTorczon_2003}.
However, the cardinal number of a PSS is at least $n+1$, where $n$ is the dimension of the problem, and one iteration may cost $\mathcal{O}(n)$ function evaluations.
The probabilistic variant consists of using randomly generated polling directions that are \textit{probabilistically descent}. Such property relies on the existence of a direction that makes an acute angle with the negative gradient with a sufficiently large probability, conditioned on the history of iterations~\cite{SGratton_et_al_2015}.
An example of such a set is a uniformly drawn direction on the unit ball and its negative. In this case, the cost of an iteration is at most~$2$ function evaluations, leading to significant gains in efficiency~\cite{SGratton_et_al_2015}.
In summary, directional direct-search methods are slow when using deterministic directions (although faster when randomized), but robust and parallelizable.
Moreover, these methods work for non-smooth and possibly noisy functions.

Among directional methods for DFO, one also finds those based on line-search schemes. In~\cite{JKiefer_JWolfowitz_1952,JCSpal_1992}, the authors considered fixed predetermined stepsize strategies coupled with directions computed via (possibly randomized) finite differences (FD) approximations to the gradient of the objective function. Methods with adaptive stepsize choices have also been proposed~\cite{ASBerahas_RHByrd_JNocedal_2019,ASBerahas_LCao_KScheinberg_2021,DMBortz_CTKelley_1998}. These methods utilize either a simplex gradient approximation or a FD approximation to the gradient of the objective function (leading to~$\mathcal{O}(n)$ evaluations per iteration), and enhance the search direction by applying a quasi-Newton scheme. Some form of sufficient decrease condition is imposed to restrict the size of the step, and possibly a curvature condition to avoid short steps. In the noisy setting, approximating the gradient can be problematic. To mitigate this issue, regression techniques can be used to compute the gradient approximation~\cite{DMBortz_CTKelley_1998}. Alternatively, optimal (in terms of minimizing the approximation error) FD~approximations can be computed if the noise level is known, or can be estimated~\cite{JJMore_SMWild_2011}, and incorporated in a FD~scheme~\cite{JJMore_SMWild_2012}. The authors in~\cite{ASBerahas_RHByrd_JNocedal_2019} propose a FD quasi-Newton method with explicit noise estimation~\cite{JJMore_SMWild_2011} and a relaxed sufficient decrease condition that is robust in the noisy setting, and that avoids re-estimating the noise at every iteration. These methods are moderately efficient and potentially scalable in the smooth case. However, when the objective function under consideration is non-smooth, such line-search methods are no longer suitable.

The other most commonly followed paradigm in the development of DFO algorithms relies on building models using objective function samples for use in trust-region methods~\cite{ARConn_PhLToint_1996,ARConn_KScheinberg_LNVicente_2009,GFasano_JLMorales_JNocedal_2009,MJDPowell_2008,DWinfield_1973}. Models can be built via interpolation or regression techniques, using basis functions such as quadratic polynomials or radial basis functions. Sample points can result from trust-region steps, be computed from well-designed deterministic model-improving techniques, or obtained from random sampling.
At each iteration of a trust-region method~\cite{ARConn_NIMGould_PhLToint_2000},
one typically considers the minimization of a quadratic model in a region around the current iterate, from which a trial point is obtained.
Accepting the trial point and updating the trust-region radius are based on the ratio between the decrease achieved in the function and the one attained at the model.
The model serves as a local approximation of the function curvature, and
in the DFO case the quality of the model depends on the geometry of the sample points~\cite{ARConn_KScheinberg_LNVicente_2008}.
DFO trust-region methods are efficient and robust when the dimension of the problems is small. However, these methods do not scale well, not only because the overall computing time becomes an issue due to the dense linear algebra of the interpolation, but also due to the ill conditioning (poor geometry) of the sample sets.
Moreover, these methods are mainly designed for smooth functions and are not easily parallelizable.
Trust-region methods for DFO have been extended to handle noise in the objective function~\cite{SCBillups_JLarson_PGraf_2013,GDeng_MCFerris_2006} (see~\cite{JLarson_MMenickelly_SWild_2019} for a review of recent developments on stochastic functions).

There have been approaches in the literature hybridizing model-based and direct-search methods. While numerical results have been promising, no comprehensive convergence theory has been presented (other than the use of the optional search step~\cite{AJBooker_etal_1998} in direct search). Pioneer examples include~\cite{ALCustodio_HRocha_LNVicente_2010,ALCustodio_LNVicente_2004,MADiniz-Ehrhardt_JMMartinez_MRaydan_2006}
and more recent ones are~\cite{CAudet_etal_2018,ARConn_SLeDigabel_2013}, and the reader is referred to~\cite[Section~2.3]{JLarson_MMenickelly_SWild_2019} for a summary of this topic.

\subsection{A rigorous DFO method that is robust for all function types}

As we have seen in Section~\ref{sec:adv}, most of the existing DFO methods have been designed and tailored for a specific type of problems, and one has to carefully choose the class to use for the best results. In this paper, we introduce a new class of derivative-free optimization methods called \FLE{}, taking advantage of two types of iterations, with the goal of achieving sustainedly good performance across all possible function types. A first iteration type (\FullEval{}) is \textit{expensive} in function evaluations, but exhibits good performance in the smooth, non-noisy case. A second iteration type (\LowEval) is \textit{cheap} in terms of function evaluations and more appropriate in the presence of non-smoothness or/and noise.

In its general form, the \FullEval{} iteration consists of a line-search step based on an approximate gradient, and the \LowEval{} iteration consists of a direct-search step. The integration of the two iterations is done by switching from one to the other when it is deemed beneficial or necessary. The main switch is a form of detection of non-smoothness or noise in~$f$ during \FullEval{} iterations.
The new class of methods is shown to be globally convergent with appropriate rates in the smooth case, whenever \FullEval{} generates enough iterates.
It is also shown to be globally convergent in the non-smooth case (and this is assured by \LowEval{} when \FullEval{} fails to bring the approximate gradient close to zero). Failure of \FullEval{} is detected by the activation of the switch condition which compares the line-search and direct-search stepsizes. To our knowledge such type of rigorous results is novel in DFO.

Our practical implementation of \FLE{} considers a BFGS step based on a FD gradient for \FullEval{}, and a probabilistic direct-search step
based on a random vector and its negative for \LowEval{} (see Section \ref{sec:practical_implementation} for more details). The numerical results show that this practical version is robust and relatively efficient for all function types of varied smoothness, and noise origin and level.

Some authors often refer to a surrogate as a low-precision
model, and to the true problem as a full precision one. In the context of our paper, Low and
Full refer to the deployed algorithmic effort, not the quality of the evaluations and thus not related to the surrogate framework from~\cite{AJBooker_etal_1998}.

\subsection{Structure of the paper and notation}

The paper is organized in the following way. In Section~\ref{sec:methods}, we start by giving the general algorithmic description of the \FLE{} framework. We then present, in Section~\ref{sec:convergence}, the global convergence rate analysis of the method when applied to a smooth objective function. Specifically, we show that there exists a subsequence of iterates for which the gradients decay to $0$ at a sublinear rate in the non-convex case, and at a linear one in the strongly convex case. We also derive in Section~\ref{sec:convergence} a global convergence result for the non-smooth case. In Section~\ref{sec:setup}, we introduce our practical choices for the \FullEval{} and \LowEval{} iterations, as well as the other solvers used in our testing environment. Finally, the performance of \FLE{} is reported in Section~\ref{sec:numerical}, based on tests conducted on different classes of functions (smooth, non-smooth, and noisy). The paper is concluded with some remarks in Section~\ref{sec:conclusions}.

Let the set of indices corresponding to successful \FullEval{} and \LowEval{} iterations be denoted by $\mathcal{I}_{SF}$ and $\mathcal{I}_{SL}$, respectively. Similarly, the indices of unsuccessful iterations are denoted as $\mathcal{I}_{UF}$ and $\mathcal{I}_{UL}$, corresponding to the \FullEval{} and \LowEval{} iterations, respectively. The set of all iterations is denoted by $\mathcal{I}_{SF} \cup \mathcal{I}_{SL} \cup \mathcal{I}_{UF} \cup \mathcal{I}_{UL}$. All norms in this paper are Euclidian.


\section{Full-low evaluation methods}
\label{sec:methods}

The main idea of the \FLE{} methods is based on the combination of two types of steps.
The first type is \textit{expensive} in function evaluations (\FullEval), but exhibits good performance in the smooth, non-noisy case. The second type is \textit{cheaper} in function evaluations (\LowEval) and at the same time more appropriate in the presence of noise and/or non-smoothness.
For such two goals, we have in mind some form of model-based or enhanced gradient-descent step for \FullEval{}, capable of exploring curvature, and a randomized directional procedure for \LowEval{}, less sensitive to noise and qualified for delivering convergence in the non-smooth regime.
The general mechanism of the \FLE{} approach is described in Algorithm~\ref{alg:flalg}.
\FullEval{} steps are consecutively taken until a certain condition, designed to sense the presence of non-smoothnes and/or noise, is activated, after which one switches to \LowEval{} iterations. The number of consecutive \LowEval{} iterations is a user defined parameter, and can be selected as a function of the last successful \FullEval{} iteration (see Section \ref{sec:practical_implementation} for details of our practical implementation of the \FLE{} method).

\begin{algorithm}
 {
\caption{\FLE{} Algorithm}
  \label{alg:flalg}
  \textbf{Initialization:} Choose an initial iterate $x_0$. Set iteration $t_0 = \FullEval{}$.
  \begin{algorithmic}[1]
    \State \textbf{For} $k=0,1,2,\dots$
    \State \hspace{0.25cm} {\bf If} $t_k = \FullEval{}$,
    attempt to compute a $\FullEval{}$ step.
\State  \hspace{0.80cm} {\bf If} success, update $x_{k+1}$ and set $t_{k+1} = \FullEval{}$.
 \State   \hspace{0.80cm} {\bf Else},
    $x_{k+1}=x_k$ and $t_{k+1} = \LowEval{}$.
    \State \hspace{0.25cm} {\bf If} $t_k = \LowEval{}$,
    compute a $\LowEval{}$ step.
\State
    \hspace{0.80cm} Update $x_{k+1}$. Decide on $t_{k+1} \in \{ \LowEval{},  \FullEval{} \}$.
  \end{algorithmic}
  }
\end{algorithm}

In the general setting, the instance we consider for the \FullEval{} step is of line-search nature, where the search direction $p_k$ is calculated based on an approximate gradient~$g_k$. The conditions imposed on both $p_k$ and $g_k$ are stated in Section~\ref{sec:convergence}. If the \FullEval{} step is successful, the next iterate is necessarily of the form $x_k + \beta_k p_k$, where~$\beta_k$ is a positive stepsize.
As is typical in nonlinear optimization, the stepsize~$\beta_k$ is required to satisfy a sufficient decrease condition of the form
\begin{equation}
\label{eq:sdc1}
f(x_k + \beta p_k) \; \leq \; f(x_k) + c \, \beta g_k ^\top p_k,
\end{equation}
where $c \in (0,1)$ is independent of~$k$ \cite{JNocedal_SJWright_2006}. Under appropriate assumptions, condition~(\ref{eq:sdc1}) can be ensured by backtracking from a fixed stepsize until it is satisfied. By doing so, one ensures steps that are simultaneously  of restricted size and not too small.

A key modification we introduce in the \FullEval{} iteration is that we do not allow the stepsize $\beta_k$ to become too small compared to a certain function of~$\alpha_k$, the direct-search stepsize used in \LowEval{} iterations (described below). In fact we only consider a \FullEval{} iteration successful if $\beta_k$ satisfies
\begin{equation}
\label{eq:key}
\beta \; \geq \; \gamma \rho(\alpha_k),
\end{equation}
where $\gamma > 0$ is independent of $k$ and $\rho(\cdot)$ is the forcing function used in the direct-search scheme of the \LowEval{} iteration.
If we backtrack too much so that we violate~(\ref{eq:key}), the \FullEval{} step is skipped. The \FullEval{} iteration is described in Algorithm~\ref{alg:fbfgsalg}.

\begin{algorithm}
 {
\caption{\FullEval{} Iteration: Line Search}
  \label{alg:fbfgsalg}

  \textbf{Input}: Iterate $x_k$ and \LowEval{} stepsize $\alpha_k$. Backtracking parameters $\bar{\beta}>0$ and $\tau \in (0,1)$.

 \textbf{Output}: $t_{k+1}$, $x_{k+1}$, and $\alpha_{k+1}$.

  \begin{algorithmic}[1]
  \State Compute an approximate gradient $g_k$.
  \State Compute a direction $p_k$.
  \State Backtracking line-search: Set $\beta=\bar{\beta}$. \textbf{If}~(\ref{eq:key}) is false, \textbf{stop}.
  \State  \textbf{While} (\ref{eq:sdc1}) is false
  \State \hspace{0.25cm} Set $\beta = \tau \beta$.
  \State \hspace{0.25cm} \textbf{If}~(\ref{eq:key}) is false, set
     $x_{k+1}=x_k$ and $t_{k+1} = \LowEval{}$, and stop the \textbf{While} cycle.
     \State {\bf If} (\ref{eq:key}) is true,  set $\beta_k = \beta$, $x_{k+1}=x_k + \beta_k p_k$, and $t_{k+1} = \FullEval{}$.
     \Statex (Note, the $\LowEval{}$ parameter $\alpha_k$ remains unchanged, $\alpha_{k+1} = \alpha_k$; see Algorithm \ref{alg:pds}.)

  \end{algorithmic}
  }
\end{algorithm}

Our proposed \LowEval{} iteration is described in Algorithm~\ref{alg:pds}, and consists of applying one step of direct search. The set of polling directions $D_k$ is for the moment left unspecified. Similar to a line-search, at each iteration, the stepsize $\alpha_k$ is required to satisfy a sufficient decrease condition of the form
\begin{equation}
\label{eq:ds-sdc}
f(x_k + \alpha_k d_k) \; \leq \;  f(x_k) - \rho(\alpha_k),
\end{equation}
where $\rho(\cdot)$ is a (positive) forcing function satisfying the conditions given in Section~\ref{sec:convergence}.
If the iteration is successful the stepsize is kept constant or increased, otherwise it is decreased. The factors for stepsize increase/decrease must obey the theory, which is met for instance if they are independent of~$k$. An initial value of $\alpha_0 > 0$ must be supplied to the first \LowEval{} iteration.

\begin{algorithm}[H]
 {
\caption{\LowEval{} Iteration: Direct Search}
  \label{alg:pds}

   \textbf{Input}: Iterate $x_k$ and stepsize $\alpha_k$.
  Direct-search parameters $\lambda \geq 1$ and $\theta \in (0,1)$.

   \textbf{Output}: $t_{k+1}$, $x_{k+1}$, and $\alpha_{k+1}$.

  \begin{algorithmic}[1]
  \State Generate a finite set $D_k$ of non-zero polling directions.
  \State \textbf{If} (\ref{eq:ds-sdc}) is true for some $d_k\in D_k$, set $x_{k+1} = x_k + \alpha_k d_k$ and $\alpha_{k+1} = \lambda \alpha_k$.
  \State \textbf{Else}, set $x_{k+1} =  x_k$ and $\alpha_{k+1} = \theta \alpha_k$.
  \State Decide \textbf{if} $t_{k+1}=\LowEval$ \textbf{or if} $t_{k+1}=\FullEval$.
  \end{algorithmic}
  }
\end{algorithm}

In order to switch from \LowEval{} to \FullEval{}, one can compare the number of unsuccessful consecutive \LowEval{} iterations ($nu_k$) to the number of line-search backtracks ($nb_{j_k}$) done in the previous \FullEval{} iteration ($j_k$). When $nu_k$ becomes greater than $nb_{j_k}$, that is perhaps a sign that too much effort was put in the current \LowEval{} iteration, a switch is desirable. We will return to this in Section~\ref{sec:setup}. For the purpose of the convergence theory, we will assume an infinity of \FullEval{} or \LowEval{} iterations whenever necessary.

\section{Convergence and rates of convergence of full-low evaluation methods}
\label{sec:convergence}

Convergence analysis for the \FLE{} methods is presented for both the smooth and non-smooth cases. The noisy case will be addressed within the non-smooth analysis; see Section \ref{sec:convergence_nonsmooth}.


\subsection{Convergence rates in the smooth case}

In this section, we analyze the behavior of the class of \FLE{} methods in the smooth case. We show that if the \FullEval{} step generates an infinity of iterates, the convergence and rates of convergence guaranteed match those of deterministic gradient descent. 
We now introduce the assumptions needed for the analysis, starting by the smoothness of~$f$.

\begin{assumption}\label{ass:lip}
The function $f$ is continuously differentiable and its gradient $\nabla f$ is Lipschitz continuous with constant $L>0$.
\end{assumption}

The approximate gradient used in \FullEval{} iterations is required to satisfy the following assumption.

\begin{assumption}
\label{ass:grfd}
The approximate gradient $g_k$ computed at $x_k$ satisfies
\begin{equation} \label{eq:gk}
    \| \nabla f(x_k) -  g_k \| \; \le \; u_g \beta_k \| g_k \|,
    \end{equation}
where $ u_g > 0$ is independent of~$k$.
\end{assumption}

When approximating the gradient by finite differences, Assumption~\ref{ass:grfd} can be rigorously ensured by Algorithm~\ref{alg:fastsr1} (see Proposition~\ref{prop:crit}).
Note that if $u_g \beta_k < 1$, Assumption~\ref{ass:grfd} implies that the negative gradient approximation~$-g_k$ is a descent direction:
\[
-\nabla f(x_k)^\top g_k \; = \; -g_k^\top g_k + (g_k-\nabla f(x_k))^\top g_k \; \geq \;  \|g_k\|^2 -
u_g \beta_k \|g_k\|^2 \; > \; 0.
\]
Finally, we state the assumptions on the \FullEval{} directions.

\begin{assumption}\label{ass:pk}
There exist constants $\kappa, u_p > 0$ such that
\begin{equation} \label{eq:cos}
 \cos(-g_k, p_k) \; = \; \frac{(-g_k) ^\top p_k}{\|g_k\| \|p_k\|} \; > \; \kappa
\end{equation}
and
\begin{equation} \label{eq:grad}
\|g_k\| \; \leq \; u_p \norm{p_k}.
\end{equation}
\end{assumption}

Condition~(\ref{eq:cos}) is classical in line-search methods~\cite[Section~3.2]{JNocedal_SJWright_2006}, and imposes an acute angle between the direction and the approximate negative gradient, bounded away from ninety degrees. Note that if $u_g \beta_k < \kappa$, conditions~(\ref{eq:gk}) and~(\ref{eq:cos}) together imply that~$p_k$ is a descent direction:
\[
-\nabla f(x_k)^\top p_k \; = \; -g_k^\top p_k + (g_k-\nabla f(x_k))^\top p_k \; \geq \; \kappa \|g_k\| \|p_k\| -
u_g \beta_k \|g_k\| \|p_k\| \; > \; 0.
\]
Condition~(\ref{eq:grad}) is also mild.
In a Newton or quasi-Newton context $p_k=-B_k^{-1}g_k$, it is essentially saying that the Hessian or secant matrix~$B_k$ is bounded (which is what is imposed in trust-region methods). One can relax~(\ref{eq:grad}) to $\|g_k\| \leq k^a \norm{p_k}$, with $a > 0$, although the sublinear rate $k^{-1/2}$ would then increase to $k^{-1/(2+a)}$.
We start the analysis by establishing a lower bound on the stepsize $\beta_k$.

\begin{lemma}
\label{lm:1}
Suppose that Assumptions~\ref{ass:lip}--\ref{ass:pk} hold. If $k$ is a successful \FullEval{} iteration, then
\[
\beta_k \; \geq \; \min\left \{\bar{\beta}, \frac{2 \tau (1 - c)}{ 2 u_g u_p + L }~\frac{ (-g_k)^\top p_k}{\norm{p_k}^2}\right\}.
\]
\end{lemma}

\begin{proof}
Using a Taylor expansion and by Assumption \ref{ass:lip}, one has
\begin{align}
    f(x_k + \beta p_k) & \leq f(x_k) + \beta p_k^\top \nabla f(x_k) + \frac{\beta^2}{2} L \norm{p_k}^2 \nonumber \\
    & \leq f(x_k) + \beta p_k^\top ((\nabla f(x_k)-g_k) + g_k) + \frac{\beta^2}{2} L \norm{p_k}^2 \nonumber \\
    & \leq f(x_k) + \beta g_k^\top p_k + \left( \frac{1}{2} L + u_g u_p\right)\beta^2 \norm{p_k}^2, \label{eq:1}
\end{align}
where the last line follows from applying~(\ref{eq:gk}) and~(\ref{eq:grad}).

If $\beta_k \neq \bar{\beta}$, then $\beta = \beta_k / \tau$ does not satisfy the sufficient decrease condition, which means
\begin{equation} \label{eq:2}
f(x_k + (\beta_k/\tau) p_k) - f(x_k) \; > \; c \, (\beta_k/\tau) g_k^\top p_k.
\end{equation}
We can now plug $\beta=\beta_k/\tau$ into~(\ref{eq:1}), and combine it with (\ref{eq:2}), to obtain
\[
g_k^\top p_k + \left( \frac{1}{2} L + u_g u_p \right) \frac{\beta_k}{\tau} \norm{p_k}^2 \; > \;  c \, g_k^\top p_k.
\]
We conclude that
\[
\beta_k \; > \; \frac{2 \tau (1 - c)}{ 2 u_g u_p + L }\frac{ (-g_k)^\top p_k}{\norm{p_k}^2},
\]
where $1- c > 0$.
The proof is completed by combining the case where $\beta_k = \bar{\beta}$ and the one presented above.
\end{proof}
\vspace{1ex}

Next, we prove that the minimum norm of the true gradient decays
at the appropriate sublinear rate, for iterations~$k$ in the set~$\mathcal{I}_{SF}$ of successful \FullEval{} iterations.

\begin{theorem} \label{th:main-FE}
Let Assumptions \ref{ass:lip}--\ref{ass:pk} hold. Assume that~$f$ is bounded from below (and let $f_{low}$ be a lower bound). Then,
\[
\min_{i = 0 \dots k-1} \norm{\nabla f(x_i)} \; \leq \; (\bar{\beta} u_g +1) \sqrt{ \frac{f(x_0) - f_{low}}{M}} \frac{1}{\sqrt{ns_{k}}},
\]
where $ns_{k} = \vert \mathcal{I}_{SF} \cap \{0,\ldots,k-1\} \vert$ is the number of successful \FullEval{} iterations up to iteration~$k$, and
\begin{equation} \label{eq:M}
M \; = \; c \min\left\{\bar{\beta} \frac{\kappa}{u_p}, \frac{2 \tau \kappa^2 (1 - c)}{ 2 u_g u_p + L }\right\}.
\end{equation}
\end{theorem}

\begin{proof}
Let $k \in \mathcal{I}_{SF}$.
First, using~(\ref{eq:cos}) and~(\ref{eq:grad}), we write
\begin{equation} \label{eq:useful}
\frac{((-g_k)^\top p_k)^2}{\|p_k\|^2} \; \geq \; \kappa^2 \|g_k\|^2 \quad   \text{and}\quad \quad
(-g_k)^\top p_k \; \geq \; \frac{\kappa}{u_p} \|g_k\|^2.
\end{equation}
Then, from the sufficient decrease condition~(\ref{eq:sdc1}) and
the lower bound for the stepsize established in Lemma~\ref{lm:1}, one obtains
\begin{equation}
\label{eq:Mei}
     f(x_k) - f(x_{k + 1}) \geq M \norm{g_k}^2,
\end{equation}
where $M$ is given in~(\ref{eq:M}).

By Assumption~\ref{ass:grfd}, the triangle inequality, and $\beta_k \leq \bar{\beta}$, we have
\[
\frac{1}{\bar{\beta} u_g  +1}  \norm{\nabla f(x_k)} \; \leq \; \norm{g_k}.
\]
Hence, plugging this bound into inequality~\eqref{eq:Mei} yields
\begin{equation} \label{eq:last-bound}
f(x_k) - f(x_{k+1}) \; \geq \; \frac{M}{(\bar{\beta} u_g + 1)^2} \norm{\nabla f(x_k)}^2.
\end{equation}

In both unsuccessful \FullEval{} and \LowEval{} iterations, the function value does not decrease, and thus from~(\ref{eq:last-bound})
\[
f(x_0) - f(x_k) \; \geq \; \frac{M}{(\bar{\beta} u_g +1)^2} \sum_{i \in \mathcal{I}_{SF} \cap \{0,\ldots,k-1\}}\norm{\nabla f(x_i)}^2.
\]
The proof is completed by taking the minimum from $0$ to $k-1$ on the successful \FullEval{} iterations and using the lower bound $f(x_k) \geq f_{low}$.
\end{proof}
\vspace{1ex}

\begin{remark}
One can see from Theorem~\ref{th:main-FE} that the number of successful \FullEval{} iterations needed to drive the gradient below a positive threshold~$\epsilon$ is of the order of~$\epsilon^{-2}$.

The number of \FullEval{} backtracking steps can also be bounded under an additional assumption. In fact, one has
$\beta_k = \bar{\beta} \tau^{nb_k}$, where ${nb_k}$ is the number of backtracks at \FullEval{} iteration~$k$.
By applying logs to this equation, one obtains $nb_k =  \log(\beta_k)/\log(\tau) - \log(\bar{\beta})/\log(\tau)$, where $\log(\tau)<0$.

When $p_k=g_k$, one can see from Assumption~\ref{ass:pk} that
\[
\frac{ (-g_k)^\top p_k}{\norm{p_k}^2} \; \geq \; \kappa \frac{ \|g_k\|}{\|p_k\|},
\]
and thus, from Lemma~\ref{lm:1}, $\beta_k$ is bounded from below. Hence, ${nb_k}$ is bounded by a constant.

When $p_k \neq g_k$, one needs the additional assumption $\| g_k \| \geq \ell_p \|p_k\|$, where $\ell_p$ is a positive constant, to also state that ${nb_k}$ is of the order of~1.
Such an assumption is guaranteed in a Newton or quasi-Newton approach $p_k= -B_k^{-1} g_k$ when $B_k$ is positive definite and bounded away from singularity.
\end{remark}

Suppose now that the sequence $\mathcal{I}_{SF}$ of \FullEval{} successful iterations is infinite.
From the rate established, one directly concludes that the gradient of~$f$ goes to zero for a subsequence of $\mathcal{I}_{SF}$ . However, since
the series $\sum_{i \in \mathcal{I}_{SF}}\norm{\nabla f(x_i)}^2$ is summable, we can also state $\lim_{k \in \mathcal{I}_{SF}} \nabla f(x_k) = 0$.

The convergence rate becomes linear when $f$ is strongly convex, which is the case when
there exists a positive constant $\mu >0$ such that for all $(x, y) \in \mathbb{R}^n$
\[
f(y) \; \geq \; f(x) + \nabla f(x)^\top (y-x) + \frac{\mu}{2} \|y - x\|^2.
\]
The argument uses the Polyak--Łojasiewicz (PL) inequality
\[
\|\nabla f(x_k) \|^2 \; \geq \; 2 \mu(f(x_k) - f(x_*)).
\]
Such inequality is typically deduced by combining
\[
2\mu (f(x_*) - f(x_k)) \; \geq \; 2 \mu \nabla f(x_k)^\top (x_* - x_k) + \mu^2 \|x_* - x_k\|^2,
\]
obtained using strong convexity with $x = x_k$ and $y = x_*$ (the unique minimizer of $f$), with
\[
2 \mu \nabla f(x_k)^\top (x_* - x_k) + \mu^2 \|x_* - x_k\|^2 \; \geq \; - \|\nabla f(x_k) \|^2.
\]
It remains to combine the PL inequality with~(\ref{eq:last-bound}), and recursively conclude that (when $k-1 \in \mathcal{I}_{SF}$)
\[
f(x_k) - f(x_*) \; \leq \; ( 1- \eta_\mu)^{ns_{k}} (f(x_0) - f(x_*))
\]
with $\eta_\mu = \frac{2 \mu c}{(\bar{\beta} u_g + 1)^2} \min\left(\bar{\beta} \frac{\kappa}{u_p}, \frac{2 \tau \kappa^2 (1 - c)}{ 2 u_g u_p + L }\right)$, and again $ns_{k} = \vert \mathcal{I}_{SF} \cap \{0,\ldots,k-1\} \vert$. Note that if is chosen such that $c \in (0,1/2)$, then $\eta_\mu \in (0,1)$.

\subsection{Convergence in the non-smooth case}\label{sec:convergence_nonsmooth}

The analysis in the non-smooth case is based on the failure of the \FullEval{} iterations.

\begin{assumption}\label{as:gb}
{\bf (\FullEval{} Failure)}
The sequence of gradient norms $\{\norm{g_k}\}_{k \in \mathcal{I}_{SF}}$ is bounded away from~$0$, i.e., there exists $\epsilon_g > 0$ such that for all iterations $k \in \mathcal{I}_{SF}$, one has $ \norm{g_k} > \epsilon_g$.
\end{assumption}

The switching condition~(\ref{eq:key}) allows us to prove under this assumption that the \LowEval{} iterations generate an infinite subsequence of iterates driving the direct-search parameter~$\alpha_k$ to zero.
This result requires the forcing function to satisfy Assumption~\ref{ass:rho}.

\begin{assumption}\label{ass:rho}
The function $\rho$ is positive, non-decreasing, and satisfies $\lim_{\alpha \rightarrow 0^+} \rho(\alpha)/\alpha = 0$.
\end{assumption}

One can see from the proof of Lemma~\ref{the:nsc} (below) that \FullEval{} steps are essentially considered as {\it search steps} from the direct-search perspective of the \LowEval{} iterations.

\begin{lemma}
\label{the:nsc}
Let Assumptions~\ref{ass:pk}--\ref{ass:rho} hold. Assume that the sequence of iterates~$\{ x_k \}$ is bounded. Then, there exists a point $x_*$ and a subsequence $\mathcal{K} \subset \mathcal{I}_{UL}$ of unsuccessful \LowEval{} iterates for which
\[
\begin{array}{ccc}
     \displaystyle\lim_{k \in \mathcal{K}} x_k \; = \; x_* & \text{ and } & \displaystyle \lim_{k \in \mathcal{K}} \alpha_k = 0.
\end{array}
\label{eq:xstar}
\]
\end{lemma}

\begin{proof}
First, consider that there is an infinity of iterations in $\mathcal{I}_{SF} \cup \mathcal{I}_{UF} \cup \mathcal{I}_{SL}$. These are all iterations~$k$ for which $\alpha_k$ does not decrease. By the second inequality in (\ref{eq:useful}), all successful \FullEval{} iterations $k \in \mathcal{I}_{SF}$
yield
\[
f(x_{k+1}) \; \leq \; f(x_k) - c \, (\kappa / u_p) \beta_{k} \norm{g_k}^2.
\]
From Assumption~\ref{as:gb} and the fact that $\beta_k$ satisfies $\beta_k \geq \gamma \rho(\alpha_k)$,
\begin{equation} \label{eq:dec1}
f(x_{k+1}) - f(x_k) \; \leq \;  - c \, (\kappa / u_p) \gamma \epsilon_g \rho(\alpha_k).
\end{equation}
Successful \LowEval{} iterations $k \in \mathcal{I}_{SL}$ achieve sufficient decrease,
\begin{equation} \label{eq:dec2}
f(x_{k+1}) - f(x_k) \; \leq \; - \rho(\alpha_k).
\end{equation}
Note that in \FullEval{} unsuccessful iterations $k \in \mathcal{I}_{UF}$ neither $x_k$ nor $\alpha_k$ changes.

Hence, given that for unsuccessful \LowEval{} iterations ($\mathcal{I}_{UL}$) the function does not decrease, we can sum from $0$ to $k \in \mathcal{I}_{SF} \cup \mathcal{I}_{UF} \cup \mathcal{I}_{SL}$ the inequalities~(\ref{eq:dec1}) and~(\ref{eq:dec2}) to obtain
\[
f(x_0) - f(x_{k+1}) \; \geq \;  c \, \gamma \, (\kappa / u_p) \epsilon_g \sum_{i \in \mathcal{I}_{SF}} \rho(\alpha_i) + \sum_{i \in \mathcal{I}_{SL}} \rho(\alpha_i).
\]
By the boundedness (from below) of~$f$, we conclude that the series is summable, which implies
$\lim_{k \in \mathcal{I}_{SF} \cup \mathcal{I}_{UF} \cup \mathcal{I}_{SL}} \rho(\alpha_k) = 0$. In conjunction with Assumption~\ref{ass:rho}, this in turn implies $\lim_{k \in \mathcal{I}_{SF} \cup \mathcal{I}_{UF} \cup \mathcal{I}_{SL}} \alpha_k = 0$.

It remains to consider the iterations in $\mathcal{I}_{UL}$. For each $k \in \mathcal{I}_{UL}$ corresponding to an unsuccessful \LowEval{} iteration, consider the previous iteration $k'=k'(k) \in \mathcal{I}_{SL} \cup \mathcal{I}_{UF} \cup \mathcal{I}_{SL}$ with $k' < k$
($k'$ could be zero). The direct-search stepsize can then be written as $\alpha_k = \theta^{k - k'} \alpha_{k'}$. Since $k' \to \infty$ and $\tau\in (0,1)$, one obtains $\alpha_k \to 0$ for all~$k \in \mathcal{I}_{UL}$.

We have thus proved that $\alpha_k$ goes to zero for all~$k$. Since $\alpha_k$ is only decreased in
unsuccessful \LowEval{} iterations, there must be an infinite subsequence of those.
From the boundedness of the sequence of iterates, one can extract a subsequence $\mathcal{K}$ of that subsequence satisfying the statement of the lemma.
\end{proof}
\vspace{1ex}

The main theorem uses the notion of generalized Clarke derivative~\cite{FHClarke_1990} at $x$ along a direction~$d \in \mathbb{R}^n$
\[
f^{\circ}(x;d) \; = \; \lim_{y \to x} \sup_{t \downarrow 0} \frac{f(y + t d) - f(y)}{t},
\]
which is well defined if $f$ is Lipschitz continuous around the point $x$.
Establishing that there is a limit point which is Clarke stationary requires the density of the so-called refining directions to be dense in the unit sphere. The proof follows the arguments
in~\cite{CAudet_JEDennis_2002,CAudet_JEDennis_2006,LNVicente_ALCustodio_2012}.

\begin{theorem} \label{th:Clarke-stat}
Let Assumptions~\ref{ass:pk}--\ref{ass:rho} hold. Assume that the sequence of iterates~$\{ x_k \}$ is bounded.
Let the function $f$ be Lipschitz continuous around the point $x_*$ defined in Lemma~\ref{eq:xstar}. Let the set of limit points of
\begin{equation} \label{eq: ref-dir}
\left\{ \frac{d_k}{\norm{d_k}}, \; d_k \in D_k, k \in \mathcal{K} \right\}
\end{equation}
be dense in the unit sphere, where $\mathcal{K} \subset \mathcal{I}_{UL}$ is given in Lemma~\ref{the:nsc}.

Then, $x_*$ is a Clarke stationary point, i.e., $f^{\circ}(x_*;d) \geq 0$ for all $d \in \mathbb{R}^n$.
\end{theorem}

\begin{proof}
Let $\bar{d}$ be a limit point of~(\ref{eq: ref-dir}), identified for a certain subsequence $\mathcal{L} \subseteq \mathcal{K}$. Then, from basic properties of the generalized Clarke derivative, and $k \in \mathcal{L}$,
\begin{align*}
    f^{\circ}(x_*;\bar{d}) &  = \lim_{x_k \to x_*} \sup_{\alpha_k \downarrow 0} \frac{f(x_k + \alpha_k \bar{d}) - f(x_k)}{\alpha_k}\\
    & \geq \lim_{x_k \to x_*} \sup_{\alpha_k \downarrow 0} \left\{ \frac{f(x_k + \alpha_k d_k) - f(x_k)}{\alpha_k} - L_f^* \| d_k - \bar{d} \| \right\} \\
    & = \lim_{x_k \to x_*} \sup_{\alpha_k \downarrow 0} \left\{ \frac{f(x_k + \alpha_k d_k) - f(x_k)}{\alpha_k}
     + \frac{\rho(\alpha_k)}{\alpha_k} \right\},
\end{align*}
where $L_f^*$ is the Lipschitz constant of $f$ around $x_*$.
Since $k \in \mathcal{L}$ are unsuccessful \LowEval{} iterations, it follows that $f(x_k + \alpha_k d_k) > f(x_k) - \rho(\alpha_k)$ which implies that
\[
\lim_{x_k \to x_*} \sup_{\alpha_k \downarrow 0} \frac{f(x_k + \alpha_k d_k) - f(x_k) + \rho(\alpha_k)}{\alpha_k} \; \geq \; 0.
\]
From this and Assumption~\ref{ass:rho}, we obtain $f^{\circ}(x_*;\bar{d}) \geq 0$.
Given the continuity of $f^{\circ}(x_*;\cdot)$, one has for any $d \in \mathbb{R}^n$
such that $\|d\|= 1$,
$f^{\circ}(x_*;d) = \lim_{\bar{d} \to d} f^{\circ}(x_*;\bar{d}) \geq 0$.
\end{proof}

\subsection{More on the smooth case (use of finite difference gradients)}

Let us return to the smooth case to clarify the imposition of Assumption~\ref{ass:grfd}. Such an assumption is related to the satisfaction of the so-called criticality step in DFO trust-region methods~\cite{ARConn_KScheinberg_LNVicente_2009,ARConn_KScheinberg_LNVicente_2009b} when using fully linear models.
In the context of the line-search method used in the \FullEval{} iterations, those models correspond to using a finite difference (FD) scheme to compute the approximate gradient~$g_k$.

The $i$-th component of the forward FD approximation of the gradient at $x_k$ is defined as
\begin{equation} \label{eq:FD}
[\nabla_{h_k} f(x_k)]_i \; = \; \frac{f(x_k + h_k e_i) - f(x_k)}{h_k}, \quad i=1,\ldots,n,
\end{equation}
where $h_k$ is the finite difference parameter and $e_i \in \mathbb{R}^n$ is the i-th canonical vector. Computing such a gradient approximation costs $n$ function evaluations per iteration, and it is implicitly assumed that such evaluations can be made.
By using a Taylor expansion, the error in the FD gradient (in the smooth and noiseless setting) can be shown to satisfy
\begin{equation} \label{eq:FDbound}
\|\nabla f(x_k) - \nabla_{h_k} f(x_k) \| \; \leq \; \frac{1}{2} \sqrt{n} \, L \, h_k.
\end{equation}

It becomes then clear that one way to ensure Assumption~\ref{ass:grfd} in practice, when $g_k = \nabla_{h_k} f(x_k)$, is to enforce $h_k \leq u_g' \beta \|\nabla_{h_k} f(x_k)\|$, for some $u_g'>0$, in which case $u_g = \frac{1}{2} \sqrt{n} \, L \, u_g'$.
Enforcing such a condition is expensive but can be rigorously done through a criticality-step type argument (see Algorithm~\ref{alg:cric}).

\begin{algorithm}[H]
 {
\caption{Criticality step: Performed if $h_k > u_g' \beta \norm{\nabla_{h_k} f(x_k)}$}
  \label{alg:fastsr1}
 {\bf Input:} $h_k$, $\nabla_{h_k} f(x_k)^{(0)} = \nabla_{h_k} f(x_k)$, $\beta$, and $\omega \in (0,1)$. Let $j=0$.

 \textbf{Output}: $\nabla_{h_k} f(x_k) = \nabla_{h_k} f(x_k)^{(j)}$ and $h_k$.

\begin{algorithmic}[1]
 \State \textbf{While} $h_k > u_g' \beta \norm{\nabla_{h_k} f(x_k)^{(j)}}$ \textbf{Do}
 \State \hspace{0.25cm} Set $j = j + 1$ and $h_k = \omega^j u_g' \beta \norm{\nabla_{h_k} f(x_k)^{(0)}}$.
    \State \hspace{0.25cm} Compute the FD approximation $\nabla_{h_k} f(x_k)^{(j)}$ using~(\ref{eq:FD}).
  \end{algorithmic}
  \label{alg:cric}
  }
\end{algorithm}

Proposition~\ref{prop:crit} shows that Algorithm~\ref{alg:cric} terminates in a finite number
of steps.

\begin{proposition}\label{prop:crit}
Let Assumption~\ref{ass:lip} hold.
Condition $h_k \leq u_g' \beta \|\nabla_{h_k} f(x_k)\|$ can be attained in a finite number of steps using Algorithm~\ref{alg:cric} if $ \norm{\nabla f(x_k)} \neq 0$.
\end{proposition}

\begin{proof}
Let us suppose that the algorithm loops infinitely. Then, for all $j \geq 1$,
using Step~3 and the satisfaction of the while--condition in Step~1,
\begin{equation} \label{eq:crit1}
     \norm{\nabla_{h_k} f(x_k)^{(j)}} \; \leq \;  \omega^j \norm{\nabla_{h_k} f(x_k)^{(0)}}.
\end{equation}
On the other hand, for all $j \geq 1$, the FD bound~(\ref{eq:FDbound}), followed by Step~3, gives us
\begin{equation} \label{eq:crit2}
\norm{\nabla f(x_k) - \nabla_{h_k} f(x_k)^{(j)}} \; \leq \; \frac{1}{2} \sqrt{n}  L \,  \omega^j u_g' \beta_k\norm{\nabla_{h_k} f(x_k)^{(0)}}.
\end{equation}
Hence, using (\ref{eq:crit1})--(\ref{eq:crit2}), we have
\[
\norm{\nabla f(x_k)} \; \leq \; \norm{\nabla f(x_k) - \nabla_{h_k} f(x_k)^{(j)}} + \norm{\nabla_{h_k} f(x_k)^{(j)}} \; \leq \; \left( \frac{\sqrt{n} L u_g' \beta_k}{2}+ 1\right) \omega^j \norm{\nabla_{h_k} f(x_k)^{(0)}}.
\]
By taking limits (and noting that $\omega \in (0,1)$), we conclude that $\norm{ \nabla f(x_k)} = 0$, which yields a contradiction.
\end{proof}

\section{Numerical setup}
\label{sec:setup}

In this section, we will first present our implementation choices for the \FLE{} method. The numerical environment of our experiments is also introduced (other methods/solvers tested, test problems chosen, and performance profiles). The tests were run using MATLAB R2019b on an Asus Zenbook with 16GB of RAM and an Intel Core i7-8565U processor running at 1.80GHz.

\subsection{Our practical full-low evaluation implementation}\label{sec:practical_implementation}

Our proposed \FullEval{} line-search iteration consists of using a direction~$p_k$ of the form~$p_k = - H_k g_k$. Our choice for the approximate gradient~$g_k$ is forward FD~(\ref{eq:FD}) with $h_k$ set to the square root of Matlab's machine precision.
Our choice for $H_k$ is the Broyden-Fletcher-Goldfarb-Shanno (BFGS) quasi-Newton update \citep{CGBroyden_1970,RFletcher_1970,DGoldfarb_1970,DFShanno_1970}.
The choice of BFGS is justified by its strong performance in the presence of derivatives~\cite{JNocedal_SJWright_2006}. The combination of BFGS with FD is regarded as a useful resource for derivative-free optimization in many NLP solvers such as for instance the {\tt fminunc} Matlab routine, and as shown in \cite{ASBerahas_RHByrd_JNocedal_2019,ASBerahas_LCao_KChoromaski_KScheinberg_2021,HJMShi_MQXuan_FOztoprak_JNocedal_2021}.
Our \FullEval{} line-search iteration is described in Algorithm~\ref{alg:BFGS-FD}.

BFGS updates the inverse Hessian approximation $H_k$ using~(\ref{eq:BFGS}),
where $j_k$ is the previous \FullEval{} iteration, and $s_k$ and $y_k$ are given in~(\ref{eq:s-and-y}). In our implementation, the first \FLE{} iteration is always \FullEval{}.
In the non-convex case, the products $s_k^\top y_k$ cannot be ensured positive. In order to maintain the positive definiteness of the matrix~$H_k$, we skip the BFGS update if $s_k^\top y_k \geq \epsilon_{c} \|s_k\| \|y_k\|$ is not satisfied, for $\epsilon_{c} \in (0,1)$ independent of~$k$. In the implementation, we used $\epsilon_{c}=10^{-10}$.
The line-search follows the backtracking scheme described in Algorithm~\ref{alg:fbfgsalg}, using  $\bar{\beta} = 1$ and $\tau=0.5$. A key feature of our \FLE{} methodology that led to rigorous results (see the proof of Lemma~\ref{the:nsc}) is to stop the line-search once condition~(\ref{eq:key}) is violated. In our implementation, we used
\begin{equation} \label{eq:forcing-imp}
\gamma \; = \; 1, \quad
\rho(\alpha_k) \; = \; \min(\gamma_1, \gamma_2 \alpha_k^2), \quad \mbox{with} \quad \gamma_1=10^{-5} \quad \mbox{and} \quad \gamma_2=10^{-3}.
\end{equation}

When $k=0$, we perform a backtracking line-search using $p_0=-g_0$ (and update $t_1$ and $x_1$) as in Algorithm~\ref{alg:fbfgsalg} (with constants as in Algorithm~\ref{alg:BFGS-FD}). The initialization of $H_0$ is done as follows. If $t_1=\FullEval{}$, then we set $H_0 = (y_0^\top s_0)/(y_0^\top y_0) I$. This formula attempts to make the size of $H_0$ similar to the one of $\nabla^2 f(x_0)^{-1}$ (see~\citep{JNocedal_SJWright_2006}).
However, if $t_1=\LowEval{}$, we set $H_0=I$.

\begin{algorithm}
\caption{\FullEval{} Iteration: BFGS with FD Gradients}
  \label{alg:BFGS-FD}
 \textbf{Input}: Iterate $x_k$ with $k \geq 1$. Information $(x_{j_k},g_{j_k},H_{j_k})$ from the previous \FullEval{} iteration~$j_k$  (if $k>0$).
  Backtracking parameters $\bar{\beta}>0$ and $\tau \in (0,1)$. Other parameters $\epsilon_{c},\gamma_1,\gamma_2>0$, $\gamma = 1$.

\textbf{Output}: $t_{k+1}$ and $(x_{k+1},H_k,g_k)$. Return the number $nb_k$ of backtrack attempts.

 \begin{algorithmic}[1]
  \State Compute the FD gradient $g_k = \nabla_{h_k} f(x_k)$ using~(\ref{eq:FD}).
  \State Set
  \begin{equation} \label{eq:s-and-y}
  s_k = x_{k} - x_{j_k} \quad \mbox{and} \quad y_{k} = g_k - g_{j_k}.
  \end{equation}
  \State {\bf If} $s_k^\top y_k \geq \epsilon_{c} \|s_k\| \|y_k\|$,
  set
\begin{equation} \label{eq:BFGS}
H_{k} \; = \; \left (I - \frac{s_k y_k^\top}{y_k^\top s_k}\right) H_{j_k}  \left(I - \frac{y_k s_k^\top }{y_k^\top s_k}\right) + \frac{s_k s_k^\top }{y_k^\top s_k}.
\end{equation}
 \State {\bf Else}, set $H_k = H_{j_k}$.
  \State Compute the direction $- H_k g_k$.
  \State Perform a backtracking line-search, and update $t_{k+1}$ and $x_{k+1}$ as in Algorithm~\ref{alg:fbfgsalg}.
  \end{algorithmic}
\end{algorithm}

For the \LowEval{} iterations, we considered the set of polling directions $D_k$ to be formed by one random direction and its negative, a variant which has shown superior performance compared with deterministic direct search based on PSS (see~\cite{SGratton_et_al_2015}). The random direction is drawn uniformly on the unit sphere. Note that the cost in function evaluations (at most~$2$ per iteration) is considerably lower than in our \FullEval{} iteration, especially when~$n$ is large.
The forcing function $\rho(\alpha_k)$ used to accept polling points is set as in~(\ref{eq:forcing-imp}).
As suggested after Algorithm~\ref{alg:pds},
the switch from \LowEval{} to \FullEval{} will occur when the number~$nu_k$ of consecutive unsuccessful \LowEval{} iterations reaches the number~$nb_{j_k}$ of backtracks done in the last \FullEval{} line-search.
In the implementation, we set the stepsize updating factors to $\lambda=1/\theta=2$.
See Algorithm~\ref{alg:pDS-imp}.

\begin{algorithm}[H]
\caption{\LowEval{} Iteration: Probabilistic Direct Search}
  \label{alg:pDS-imp}
 \textbf{Input}: Iterate $x_k$ and stepsize $\alpha_k$. Direct-search parameters $\lambda \geq 1$ and $\theta \in (0,1)$. Number $nu_k$ of unsuccessful \LowEval{} iterations since last \FullEval{} iteration~$j_k$. (Set $nu_k = 0$ if $t_{k-1}=\FullEval$.) Number $nb_{j_k}$ of backtracks done at \FullEval{} iteration~$j_k$.

 \textbf{Output}: $t_{k+1}$, $x_{k+1}$, $\alpha_{k+1}$, and $nu_{k+1}$.
    \begin{algorithmic}[1]
  \State Generate $d \in \mathbb{R}^n$ uniformly on the unit sphere of $\mathbb{R}^n$, and set $D_k = [d,-d]$.
  \State Poll as in Algorithm~\ref{alg:pds}.
  \State Set $nu_{k+1} = nu_k$ in the successful case, $nu_{k+1} = nu_k + 1$ otherwise.
  \State \textbf{If} $nu_k < nb_{j_k}$, $t_{k+1}=\LowEval$.
  \State \textbf{Else}, $t_{k+1}=\FullEval$.
  \end{algorithmic}
\end{algorithm}

\subsection{Other solvers tested}

We compared the numerical performance of our implementation of \FLE{} to four other approaches:
(i) a line-search BFGS method based on FD gradients (as if there were only \FullEval{} iterations), referred to as \texttt{BFGS-FD};
(ii) probabilistic direct search (as if there were only \LowEval{} iterations), referred to as \texttt{pDS};
(iii) an interpolation-based trust-region solver, \DFOTR{}~\cite{ASBandeira_KScheinberg_LNVicente_2012}; (iv) a line-search BFGS method based on FD gradients and noise estimation, \FDLM{}~\cite{ASBerahas_RHByrd_JNocedal_2019}; (v) a mesh adaptive direct search solver, \NOMAD{}.

\DFOTR{}~\cite{ASBandeira_KScheinberg_LNVicente_2012} builds models by quadratic interpolation.
At the first iteration, the function is evaluated using a sample set of~$2n+1$ points, given by $x_0$ and $x_0\pm \Delta_0 e_i$, with $e_i$ the $i$-th canonical vector.
Until the cardinal number of the sample set reaches $p_{\max} = (n+1)(n+2)/2$,
all trust-region trial points $x_k+s_k$ are added to the sample set, and models are computed by minimum Frobenius norm interpolation~\cite{ARConn_KScheinberg_LNVicente_2009b,MJDPowell_2004}. Once there are~$p_{max}$ points in the sample set, a quadratic model is built by determined interpolation. Then, for each new trial point added, an existing sample point is discarded (the farthest away from~$x_k+s_k$).
When $\Delta_k$ falls below a certain threshold, points that are too far from the current iterate are discarded, expecting that the next iterations will refill the sample set, providing an effect similar to a criticality step.
The trust-region radius~$\Delta_k$ is updated as it is common in trust-region methods, the major difference being that it is never reduced when the sample set has less than~$n+1$ points.

\FDLM{}~\cite{ASBerahas_RHByrd_JNocedal_2019} is a linesearch BFGS method for noisy functions where the gradients are approximated using~FD. This approach differs from our \FullEval{} iterations in two main ways. First, in \FDLM{}~the noise level is estimated using~\cite{JJMore_SMWild_2011}, and such estimate is used to select the FD parameter as suggested in~\cite{JJMore_SMWild_2012}. Second, in~\cite{ASBerahas_RHByrd_JNocedal_2019}, the sufficient decrease condition is relaxed using the estimated noise level to prevent valuable points from being discarded. Moreover, to avoid estimating the noise at every iteration, the authors~\cite{ASBerahas_RHByrd_JNocedal_2019} developed a recovery procedure that estimates the noise only as needed (as the optimization progresses). The \FDLM{}~solver was chosen to compare our method to a version of FD-based BFGS where the noise is directly taken into account. We should emphasize again that in our proposed method, \FLE{}, the noise is directly handled by the \LowEval{} iterations.


\NOMAD{}~\cite{CAudet_etal_2021} is a solver that implements Mesh Adaptive Direct Search (MADS)~\citep{CAudet_JEDennis_2006} under general nonlinear constraints.
MADS, like pDS, belongs to the class of directional direct-search methods described in Section~\ref{sec:adv}, and the main differences lie
in the choice of polling directions and the decrease condition.
The polling directions belong to positive spanning sets that asymptotically cover the unit sphere densely.
Similarly to probabilistic direct search, the iterates and the mesh/step size are updated according to the outcome of the iteration,
but new points are accepted based on a simple decrease of the function.
Since its release in 2001, \NOMAD{} has been improved to account for various features and extensions such as polling strategies, surrogate functions, and parallelism.
In particular, the current default version includes a search step minimizing quadratic interpolation models based on previously evaluated poll points.




\subsection{Classes of problems tested and profiles used}\label{sec:problem_details}

We considered $62$ problems\footnote{{\tt ALLINITU}, {\tt ARGLINB}, {\tt ARGLINC}, {\tt BARD}, {\tt BEALE}, {\tt BOX3}, {\tt BRKMCC}, {\tt BROWNAL}, {\tt BROWNBS}, {\tt BROWNDEN}, {\tt CHNROSNB}, {\tt CUBE},
{\tt DECONVU}, {\tt DENSCHNA}, {\tt DENSCHNB}, {\tt DENSCHNC}, {\tt DENSCHND}, {\tt DENSCHNF}, {\tt DIXON3DQ}, {\tt DJTL}, {\tt ENGVAL2}, {\tt ERRINROS}, {\tt EXPFIT},
{\tt EXTROSNB}, {\tt GENHUMPS}, {\tt GROWTH}, {\tt GROWTHLS}, {\tt HAIRY}, {\tt HEART6LS}, {\tt HEART8LS}, {\tt  HELIX}, {\tt HILBERTA}, {\tt HILBERTB}, {\tt HIMMELBF},
{\tt HIMMELBG}, {\tt HUMPS}, {\tt JENSMP}, {\tt KOWOSB}, {\tt LOGHAIRY}, {\tt MARATOSB}, {\tt METHANB8}, {\tt MEXHAT}, {\tt MEYER3}, {\tt NONMSQRT}, {\tt PALMER1C},
{\tt PALMER1D}, {\tt PALMER2C}, {\tt PALMER3C}, {\tt PALMER4C}, {\tt PALMER4E}, {\tt PALMER5C}, {\tt PALMER5D}, {\tt PALMER6C}, {\tt PALMER7C}, {\tt PALMER8C}, {\tt ROSENBR},
{\tt SINEVAL}, {\tt SISSER}, {\tt TOINTQOR}, {\tt WATSON}, {\tt YFITU}, and {\tt ZANGWIL2}.}
from the \cuter~library~\cite{NIMGould_DOrban_PhLToint_2015} for the smooth case, with dimensions between $2$ and $51$; Table~\ref{table-small} summarizes the distribution of the problems in terms of their dimensions.

\begin{table}[H]
\centering
\begin{tabular}{|l|l|l|l|l|l|l|l|l|l|l|l|l|}
\hline
Dimension of the problem & 2  & 3 & 4 & 5 & 6 & 7 & 8  & 9 & 10 & 31 & 50 & 51 \\ \hline
Number of problems       & 21 & 9 & 5 & 1 & 2 & 1 & 10 & 1 & 5  & 2  & 4  & 1  \\ \hline
\end{tabular}
\caption{$62$ problems from~\cite{NIMGould_DOrban_PhLToint_2015}.\label{table-small}}
\end{table}

In order to test the scalability of the methods, we selected~12 problems\footnote{{\tt ARGLINA}, {\tt ARWHEAD}, {\tt  BROYDN3D}, {\tt  DQRTIC}, {\tt ENGVAL1}, {\tt FREUROTH}, {\tt PENALTY2}, {\tt NONDQUAR}, {\tt ROSENBR}, {\tt SINQUAD}, {\tt TRIDIA}, and {\tt WOODS}.} from the \cuter{} library~\cite{NIMGould_DOrban_PhLToint_2015} that represent problems with different features (non-linearity, non-convexity, partial separability) and for which one can vary the dimension~$n$. For these problems, we considered dimensions $n=40$ and $n=80$.

For the non-smooth and noisy cases, we chose~$53$ problems from the test set described in~\cite{JJMore_SMWild_2009}. The dimensions in this set vary from~2 to~12, and its distribution is summarized in Table~\ref{MW}.
The problems in~\cite{JJMore_SMWild_2009} come in three different perturbed forms: piecewise smooth, deterministic noise, and stochastic noise. In~\cite{JJMore_SMWild_2009}, the authors considered the multiplicative noise case, however, we also tested additive noise.

{\bf Additive noise.}
The function has the form $f(x) = \phi(x)+\varepsilon(x)$,
where $\phi$ is a smooth function. In the case of deterministic noise, $\varepsilon(x)= \epsilon_f \psi(x)$, where $\psi : x \in \mathbb{R}^n \to [-1,1]^n$ is a deterministic noisy function (see~\cite{JJMore_SMWild_2009} for a full description of the function $\psi$).  When considering stochastic noise, $\varepsilon(x)$ is a realization of a uniform random variable $U(-\epsilon_f, \epsilon_f)$.

{\bf Multiplicative noise.} The function takes the form $f(x) = \phi(x)(1 + \xi (x))$ where $\phi$ is a smooth function. As in the additive noise case, when the noise is deterministic $\xi (x) = \epsilon_f \psi(x)$. In the case of stochastic noise, the values $\xi (x)$ are drawn from a uniform random variable as described earlier. The multiplicative noise case can be regarded as a special case of the additive noise case as $f (x) = \phi(x) + \varepsilon(x)$, where $\varepsilon(x) = \phi(x) \xi(x)$. When the true optimal objective function value is equal to~$0$, the noise is also decaying to $0$ when converging to the optimum.

\begin{table}[H]
\centering
\begin{tabular}{|l|l|l|l|l|l|l|l|l|l|l|l|l|}
\hline
Dimension of the problem & 2  & 3 & 4 & 5 & 6 & 7 & 8  & 9 & 10 & 11 & 12  \\ \hline
Number of problems       & 5 & 6 & 5 & 4 & 4 & 5 & 6 & 5 & 4  & 4  & 5  \\ \hline
\end{tabular}
\caption{$53$ problems from~\cite{JJMore_SMWild_2009}.\label{MW}}
\end{table}

Performance profiles are a metric of comparison introduced in~\cite{EDDolan_JJMore_2002}
to assess the performance of a given set of solvers~$\mathcal{S}$ for a given set of problems~$\mathcal{P}$. They are a visual tool where the highest curve (top and left) corresponds to the solver with the best overall performance.
Let~$t_{p,s}>0$ be a performance measure of the solver $s \in \mathcal{S}$ on the problem $p \in \mathcal{P}$, which in our case was set to the number of function evaluations.
The curve for a solver~$s$ is defined as the fraction of problems where the performance ratio is at most $\alpha$,
\[
\rho_s(\alpha) \; = \; \frac{1}{|\mathcal{P}|} \text{size} \left\{ p \in \mathcal{P} : r_{p,s} \leq \alpha \right\},
\]
where the performance ratio~$r_{p,s}$ is defined as
\[
r_{p,s} \; = \; \frac{t_{p,s}}{\min \{ t_{p,s} : s \in \mathcal{S}\}}.
\] The convention $r_{p,s} = +\infty$ is used when a solver $s$ fails to satisfy the convergence test for problem~$p$.
The convergence test used is
\begin{equation}
\label{eq:convprof}
f(x_0) - f(x) \; \geq \; (1 - \tau) (f(x_0) - f_L),
\end{equation}
where $\tau > 0$ is a tolerance, $x_0$ is the starting point for the problem, and $f_L$ is computed for each problem $p \in \mathcal{P}$ as the smallest value of~$f$ obtained by any solver within a given number of function evaluations.
Solvers with the highest values of $\rho_s(1)$ are the most efficient, and those with the highest values of $\rho_s(\alpha)$, for large $\alpha$, are the most robust.

Users with expensive optimization problems are often interested in the performance of solvers as the number of function evaluations increases.
Specifically, they are interested in the percentage of problems that can be solved for a given tolerance and a given (small) budget of
evaluations.
Data profiles~\cite{JJMore_SMWild_2009} are a tool to visualize such a behavior. A data profile of a solver $s \in \mathcal{S}$ is defined as
\begin{equation}
d_s(\alpha) = \frac{1}{|\mathcal{P}|} \text{size} \left\{ p \in \mathcal{P} : \frac{t_{p,s}}{n_p +1} \leq \alpha \right\},
\label{eq:dataprof}
\end{equation}
where $n_p$ is the dimension of $p \in \mathcal{P}$, and again, $t_{p,s}$ is the number of function evaluations taken to satisfy~\eqref{eq:convprof}.
Similarly to performance profiles, $t_{p,s} = + \infty$ when the solver $s$ fails to satisfy~\eqref{eq:convprof} for problem $p$ (using a given number of function evaluations).
The budget of function evaluations in data profiles~(\ref{eq:dataprof}) is scaled by $n_p+1$, which is the number of evaluations required
to compute a simplex gradient.
While performance profiles curves reflect relative solvers performance, data profiles display each curve independently of the others.

\section{Numerical results}
\label{sec:numerical}

For performance profiles, the solvers were given a budget of $2000n$ function evaluations.
A budget of $100(n+1)$ was used for the data profiles in order to test the efficiency of the methods under small budgets of evaluations.
The optimality tolerance parameters were chosen smaller for the non-noisy cases, $\tau \in \{ 10^{-2},10^{-5} \}$, and larger for the noisy ones, $\tau \in \{ 10^{-1},10^{-3} \}$, to account for the relative noise level of the problems.

\subsection{Smooth problems}

Figure~\ref{fig:comp4} shows that \DFOTR{} (magenta downward triangles) is the most efficient method for the small smooth problems.

\begin{figure}[h]
\vskip -25ex
\hskip -20ex
\centering
\begin{minipage}[t]{0.47\textwidth}
  \includegraphics[scale=0.5]{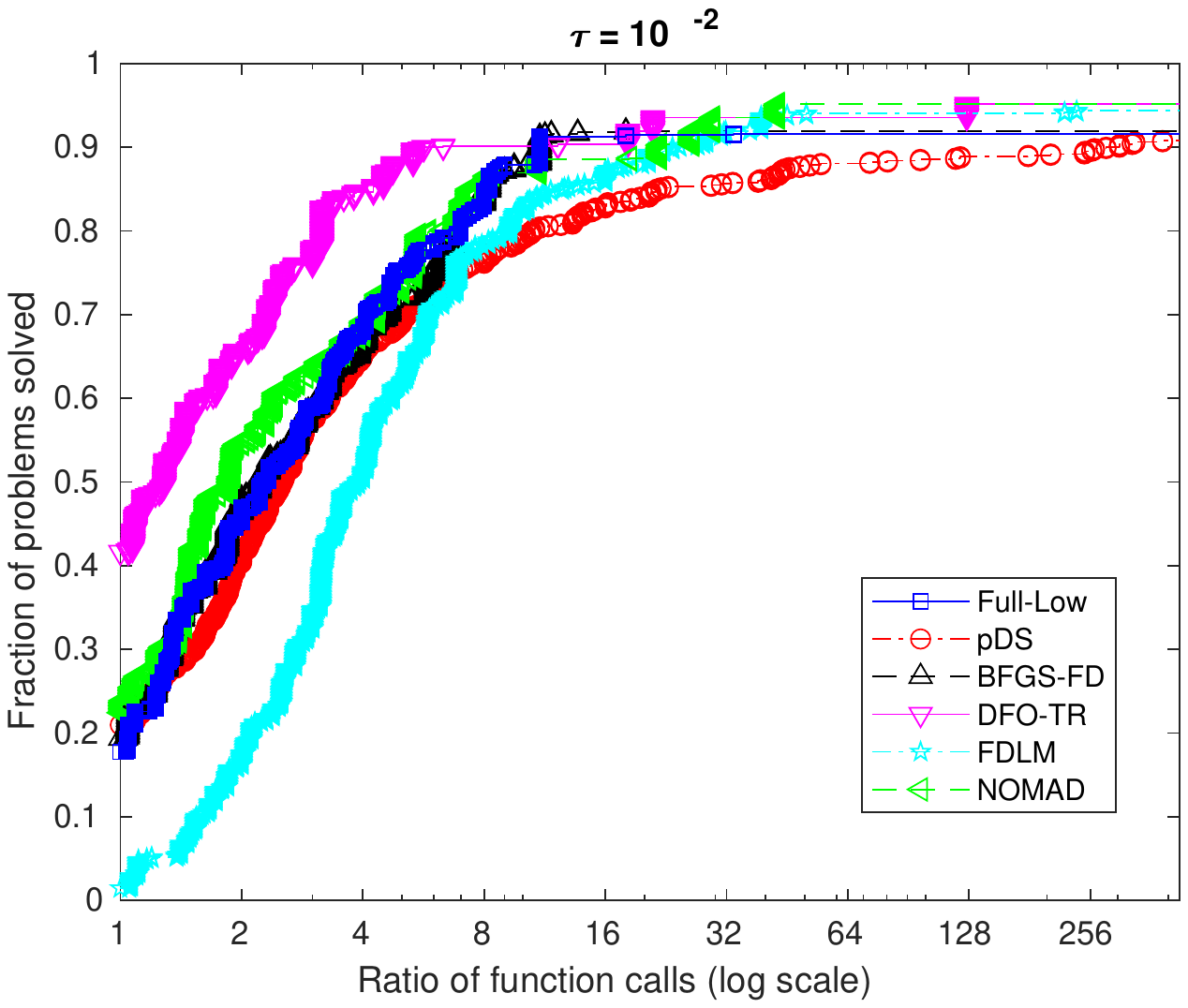}
\end{minipage}
 \hskip 5ex
\begin{minipage}[t]{0.47\textwidth}
  \includegraphics[scale=0.5]{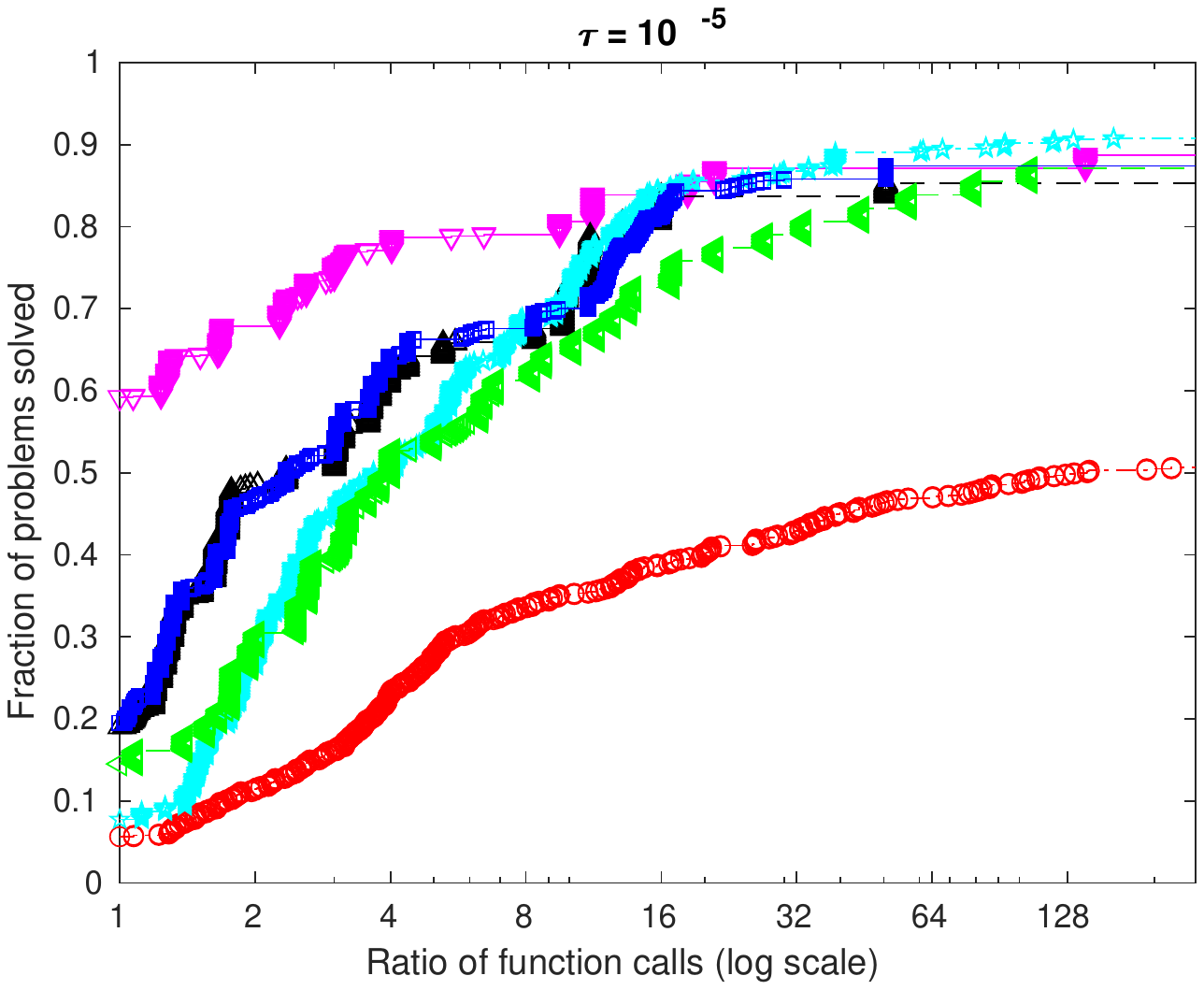}
\end{minipage}\\
\vskip -50ex
\hskip -20ex
\begin{minipage}[t]{0.47\textwidth}
  \includegraphics[scale=0.5]{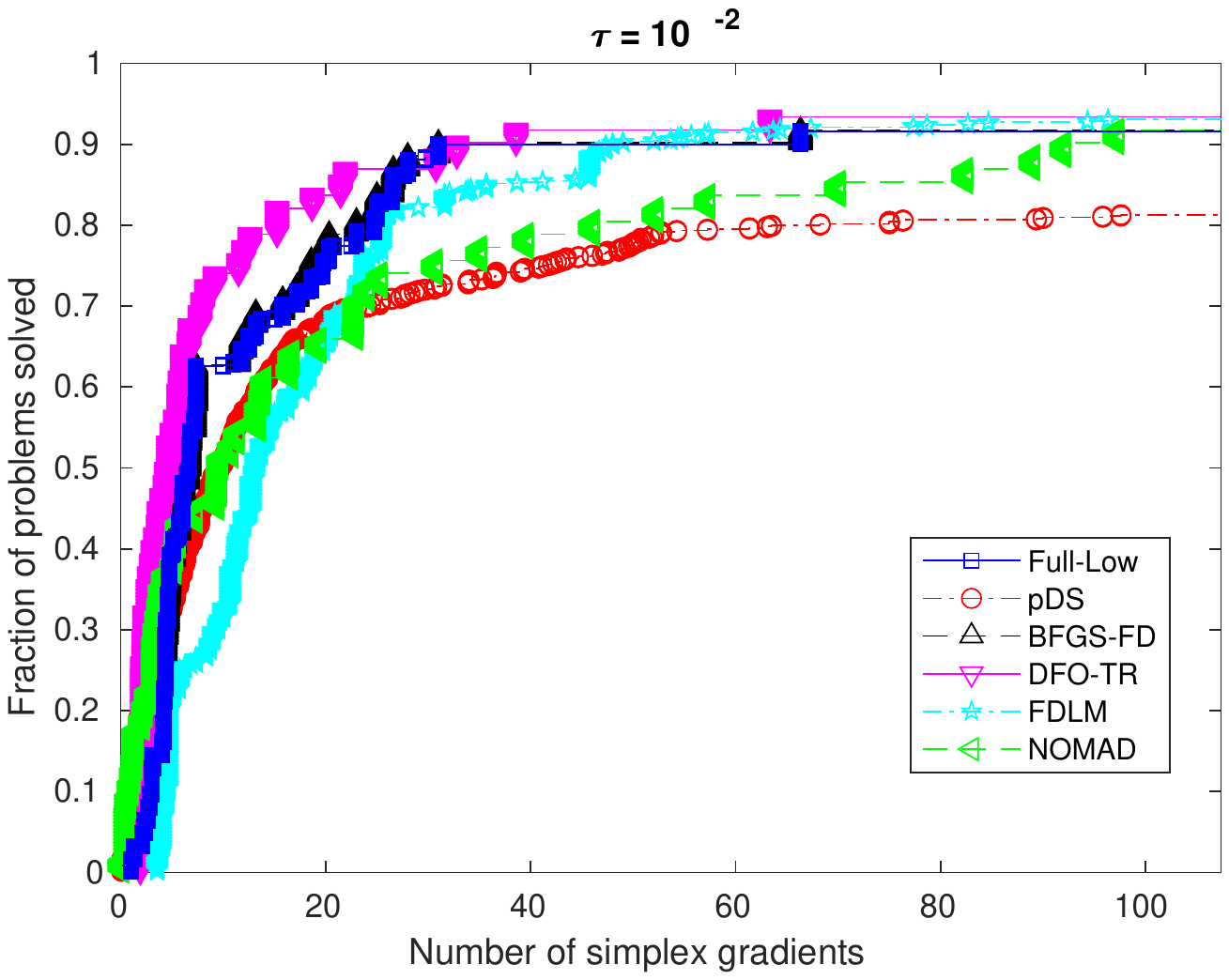}
\end{minipage}
 \hskip 5ex
\begin{minipage}[t]{0.47\textwidth}
  \includegraphics[scale=0.5]{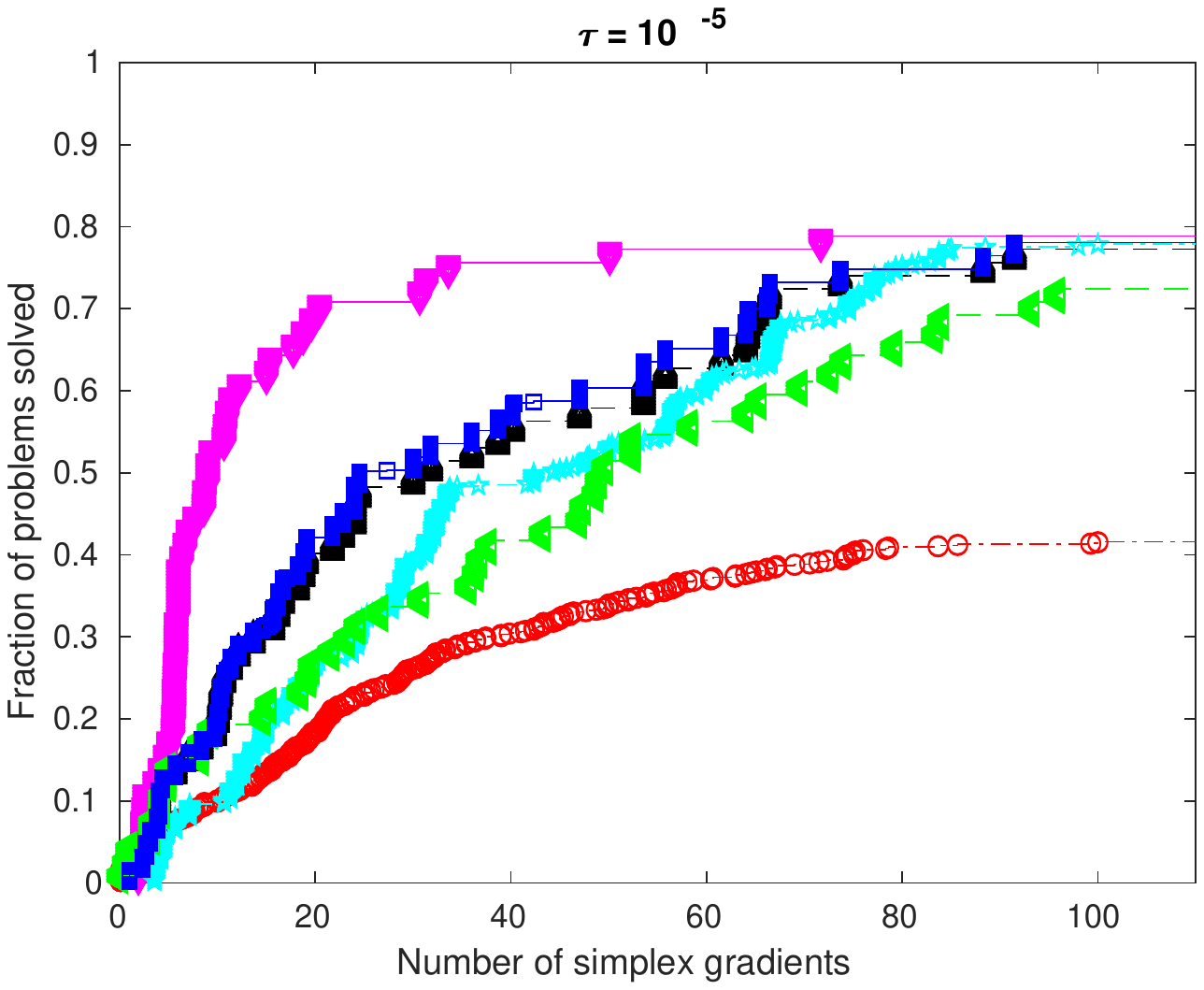}
\end{minipage}

\vskip -25ex
\caption{\centering Performance profiles (first row) and data profiles (second row) with $\tau = 10^{-2}$, $10^{-5}$ for the 6 solvers. Results for the 62 small smooth problems from~\cuter~\cite{NIMGould_DOrban_PhLToint_2015}.}
\label{fig:comp4}
\end{figure}

The \FLE{} method (blue squares) and \texttt{BFGS-FD} (black upward triangles) exhibit essentially the same performance, and this is because condition~(\ref{eq:key}) is rarely violated for smooth problems.
For low accuracy, \NOMAD{} (green leftward triangles) has a comparable performance to \FLE{} and \texttt{pDS} (red circles) does perform quite well as expected.
For the $n = 40,80$ smooth problems (see Figure~\ref{fig:scal}), the \FLE{} method appears to be both the most efficient and the most robust solver (performance similar to \texttt{BFGS-FD} as no \LowEval{} iterations seem to occur). The performance of \DFOTR{} (magenta downward triangles) deteriorates with the dimension, perhaps due to ill-conditioning of the sample set. In all the tests for smooth problems, \texttt{FDLM} (light blue stars) delivers poor efficiency but strong robustness. \NOMAD{} delivers the worst performance looking both at performance and data profiles.

\begin{figure}[H]
\vskip -25ex
\hskip -20ex
\centering
\begin{minipage}[t]{0.47\textwidth}
  \includegraphics[scale=0.5]{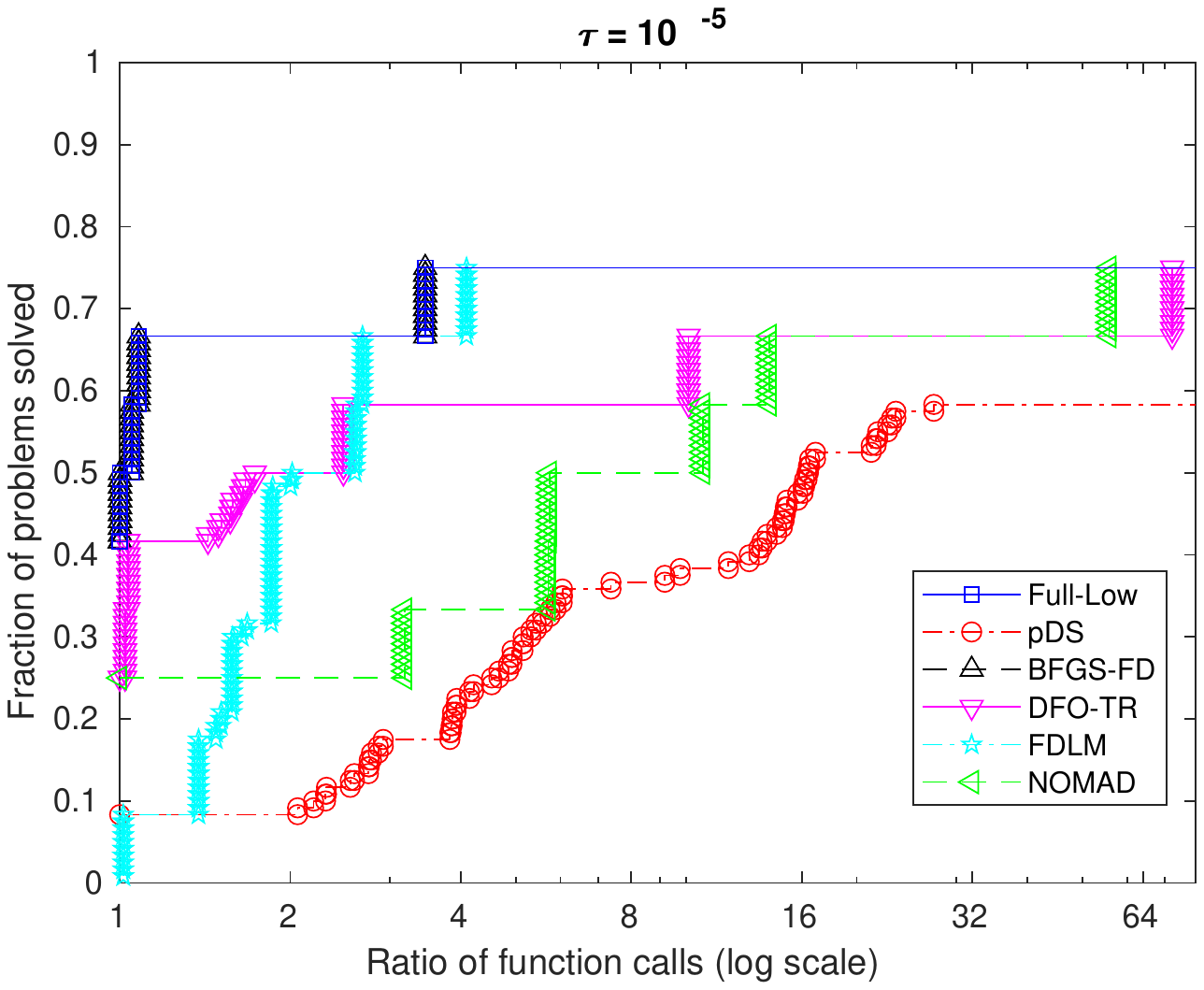}
\end{minipage}
 \hskip 5ex
\begin{minipage}[t]{0.47\textwidth}
  \includegraphics[scale=0.5]{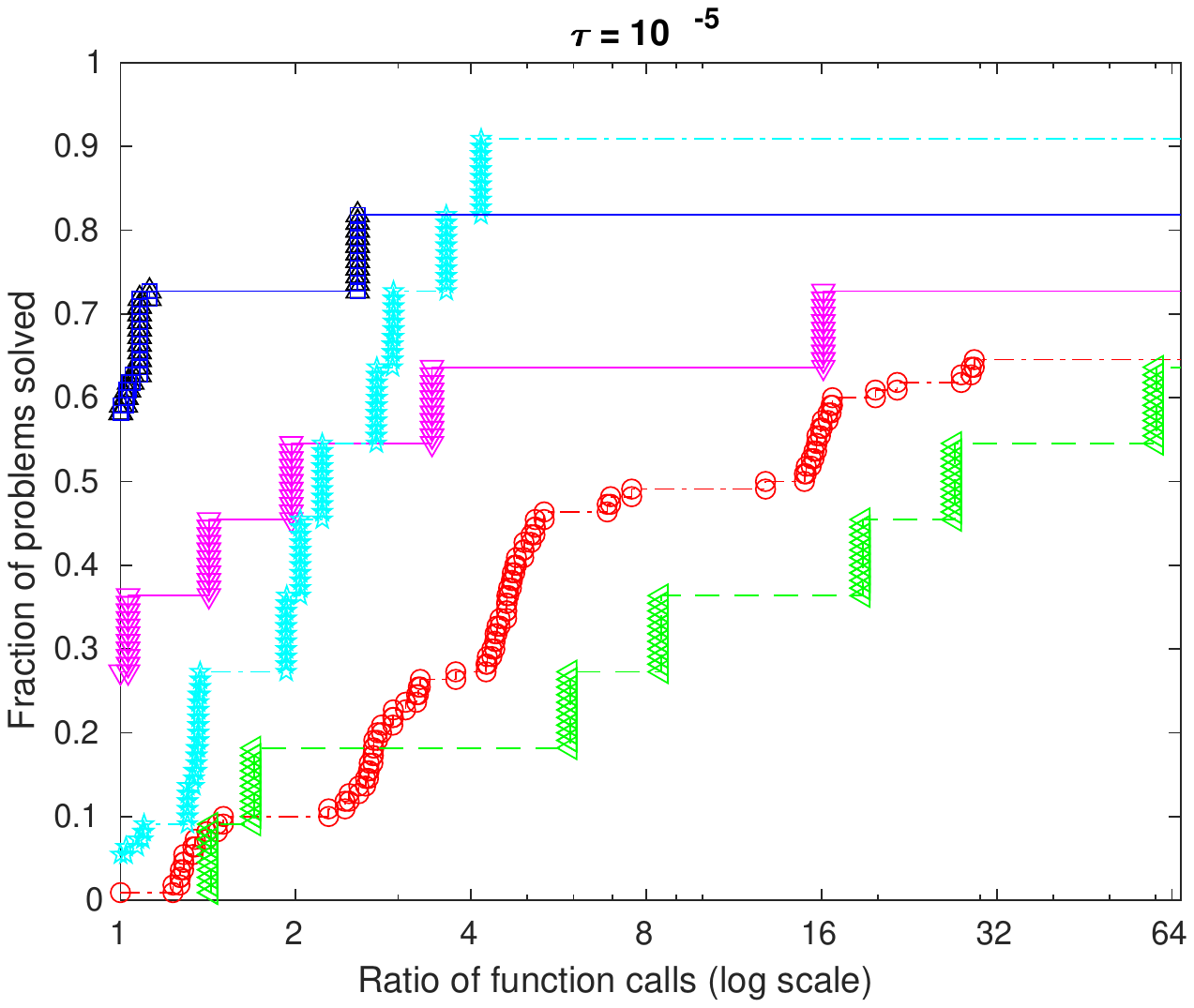}
\end{minipage}
\vskip -50ex
\hskip -21ex
\centering
\begin{minipage}[t]{0.47\textwidth}
  \includegraphics[scale=0.5]{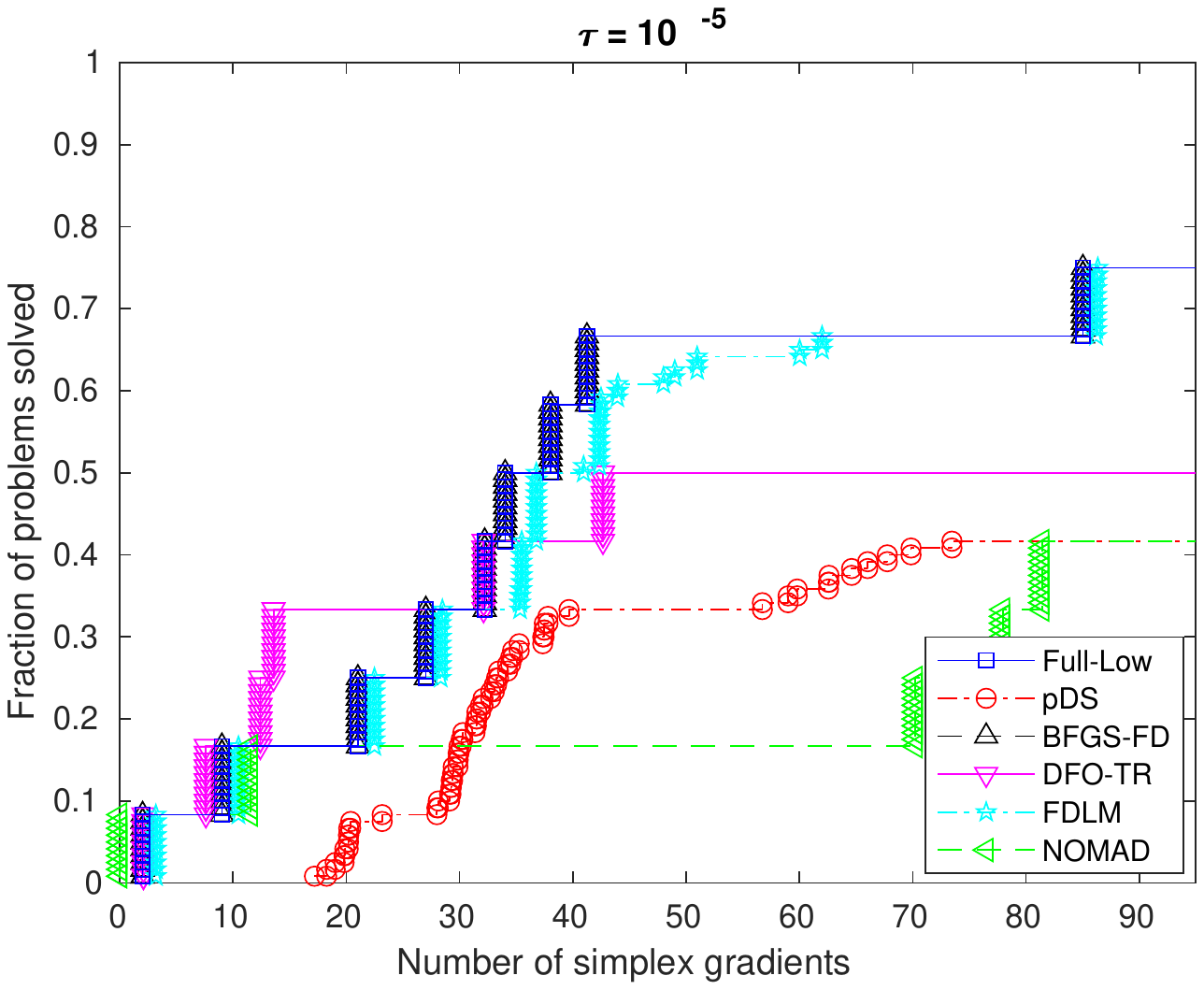}
\end{minipage}
 \hskip 5ex
\begin{minipage}[t]{0.47\textwidth}
  \includegraphics[scale=0.5]{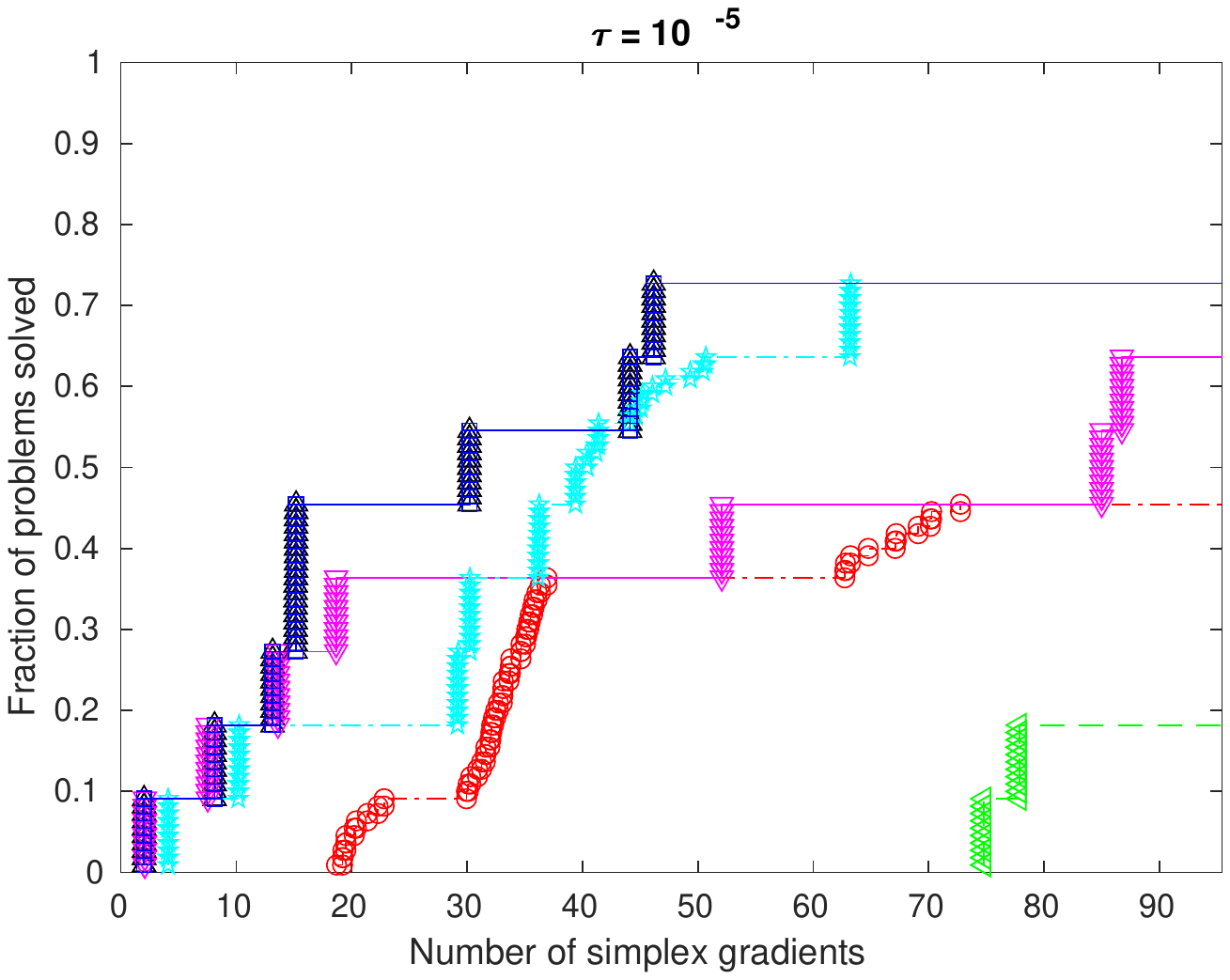}
\end{minipage}
\vskip -25ex
\caption{\centering Performance profiles (first row) and data profiles (second row) with $\tau = 10^{-5}$
for the 6 solvers. Results for the 12 not-so-small smooth problems from~\cuter~\cite{NIMGould_DOrban_PhLToint_2015} ($n=40$ on left and $n=80$ on the right).}
\label{fig:scal}
\end{figure}

\subsection{Non-smooth problems}

In Figure~\ref{fig:nondiff}, we present the results on the piecewise smooth problems. One can see that the \FLE{} curve is above all, followed by the BFGS-FD, then \NOMAD{}. On this set of problems, one can see that \FLE{} performs better that \texttt{BFGS-FD}, due to the presence of non-smoothness. In fact, the use of \LowEval{} iterations has significantly improved the performance of \texttt{BFGS-FD}. Note that \texttt{pDS} and \DFOTR{} have the worst performance, but \texttt{FDLM} is not much better than these two. Note that none of these methods were designed to handle problems with non-smoothness. Looking at data profiles, \FLE{} solves the largest number of problems for both low and high accuracies for small budgets.

\begin{figure}[H]
\vskip -25ex
\hskip -20ex
\centering
\begin{minipage}[t]{0.47\textwidth}
  \includegraphics[scale=0.5]{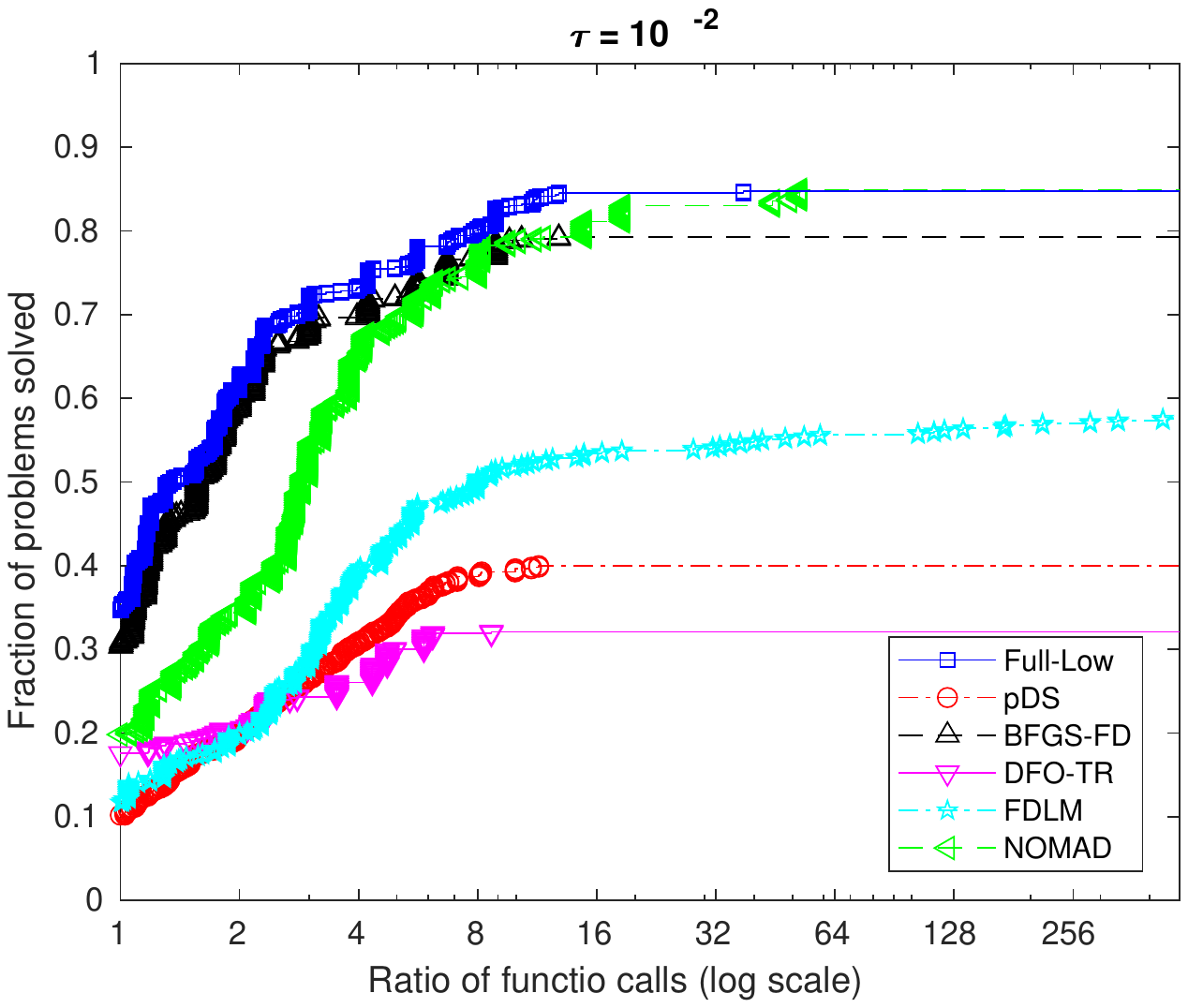}
\end{minipage}
 \hskip 5ex
\begin{minipage}[t]{0.47\textwidth}
  \includegraphics[scale=0.5]{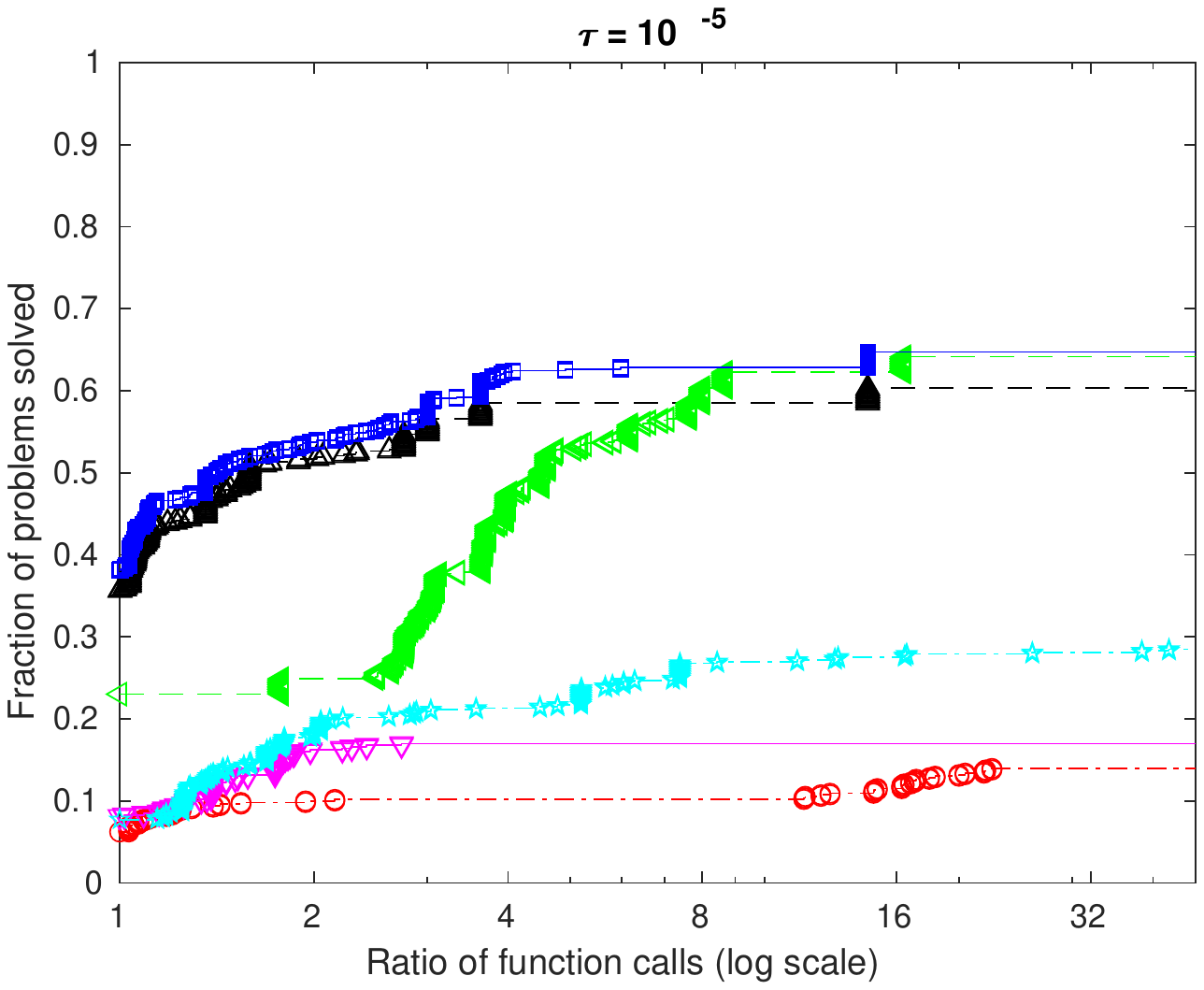}
\end{minipage}
\vskip -50ex
\hskip -20ex
\centering
\begin{minipage}[t]{0.47\textwidth}
  \includegraphics[scale=0.5]{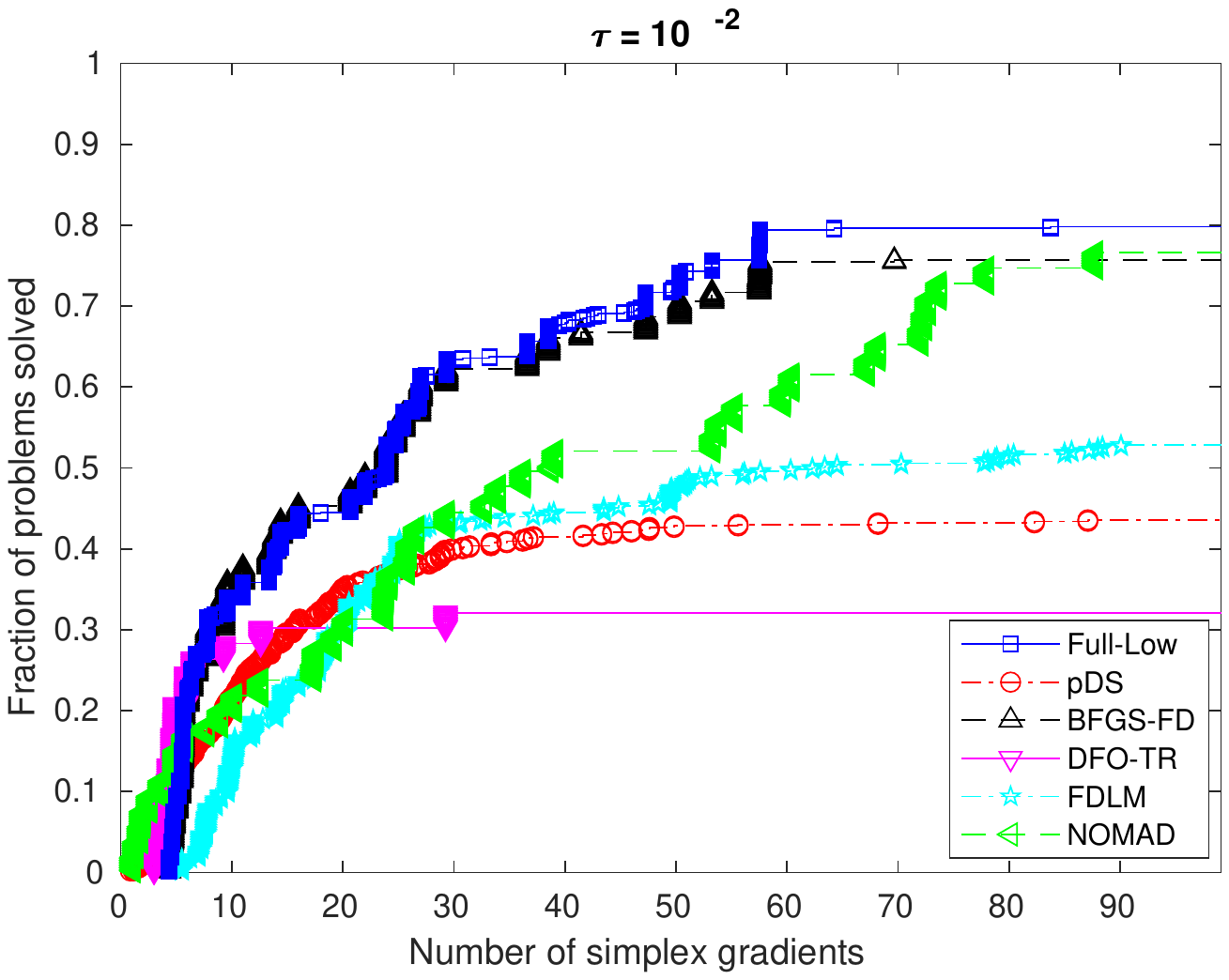}
\end{minipage}
 \hskip 5ex
\begin{minipage}[t]{0.47\textwidth}
  \includegraphics[scale=0.5]{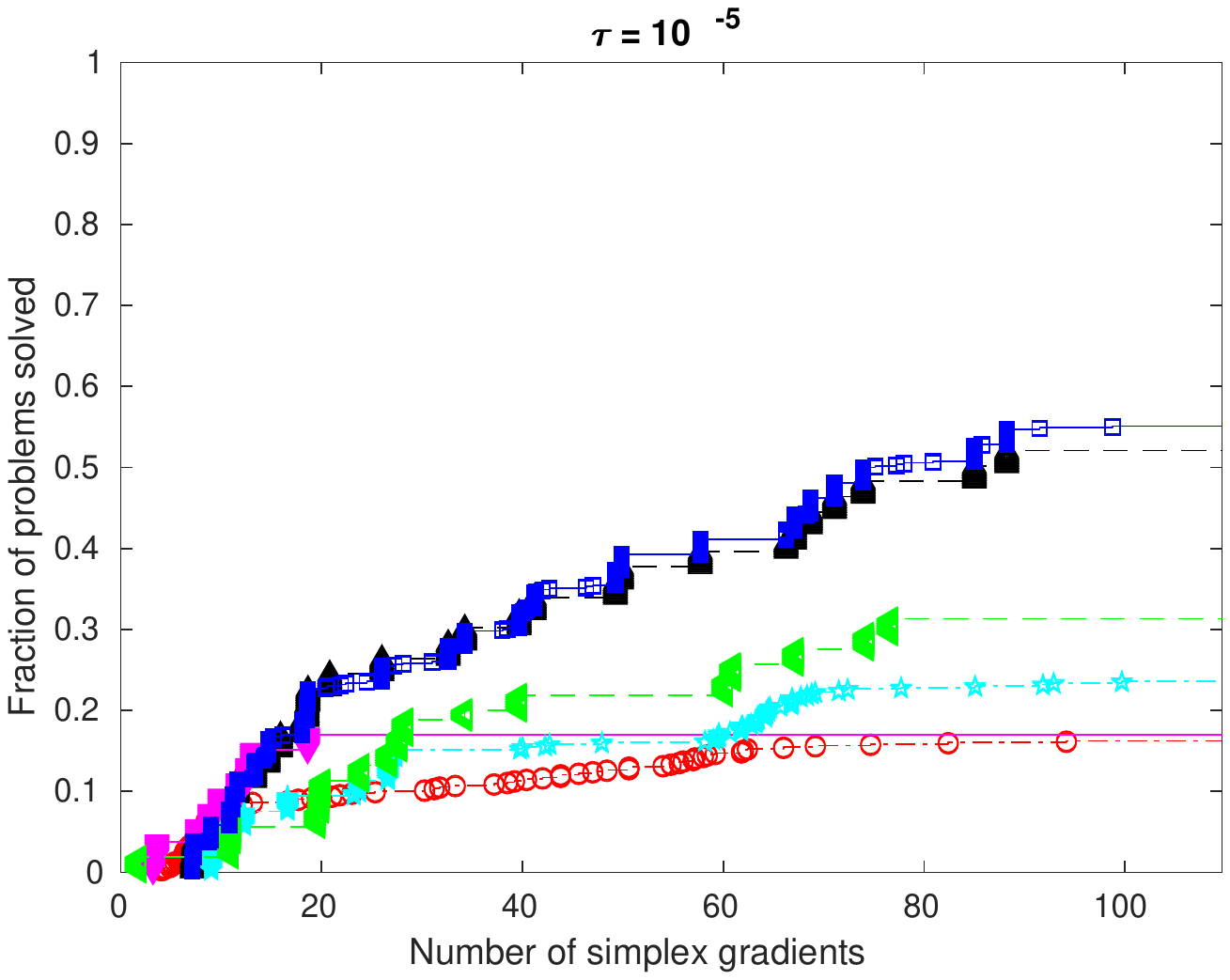}
\end{minipage}
\vskip -28ex
\caption{\centering Performance profiles (first row) and data profiles (second row) with $\tau = 10^{-2}$, $10^{-5}$ for the 6 solvers. Results for the piecewise smooth problems in~\cite{JJMore_SMWild_2009}.}
\centering
\label{fig:nondiff}
\end{figure}

\subsection{Noisy problems}

We now compare the $6$ solvers on the noisy problems described in Section~\ref{sec:problem_details}. The noise level~$\epsilon_f$ is chosen to be equal to $10^{-3}$ for both deterministic and stochastic cases.

\begin{figure}[H]
\vskip -25ex
\hskip -20ex
\centering
\captionsetup{justification=centering}
\begin{minipage}[t]{0.47\textwidth}
  \includegraphics[scale=0.5]{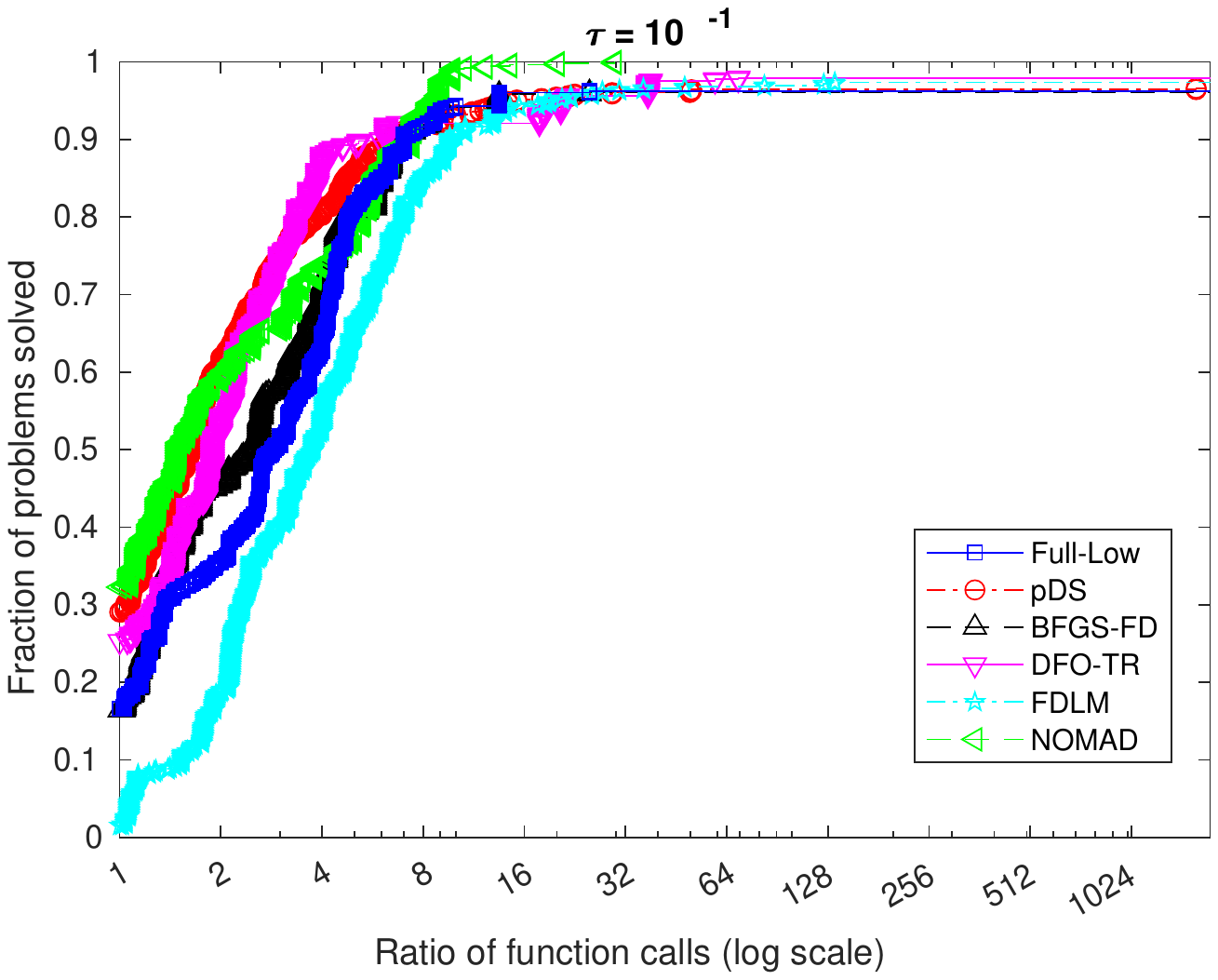}
\end{minipage}
 \hskip 5ex
\begin{minipage}[t]{0.47\textwidth}
  \includegraphics[scale=0.5]{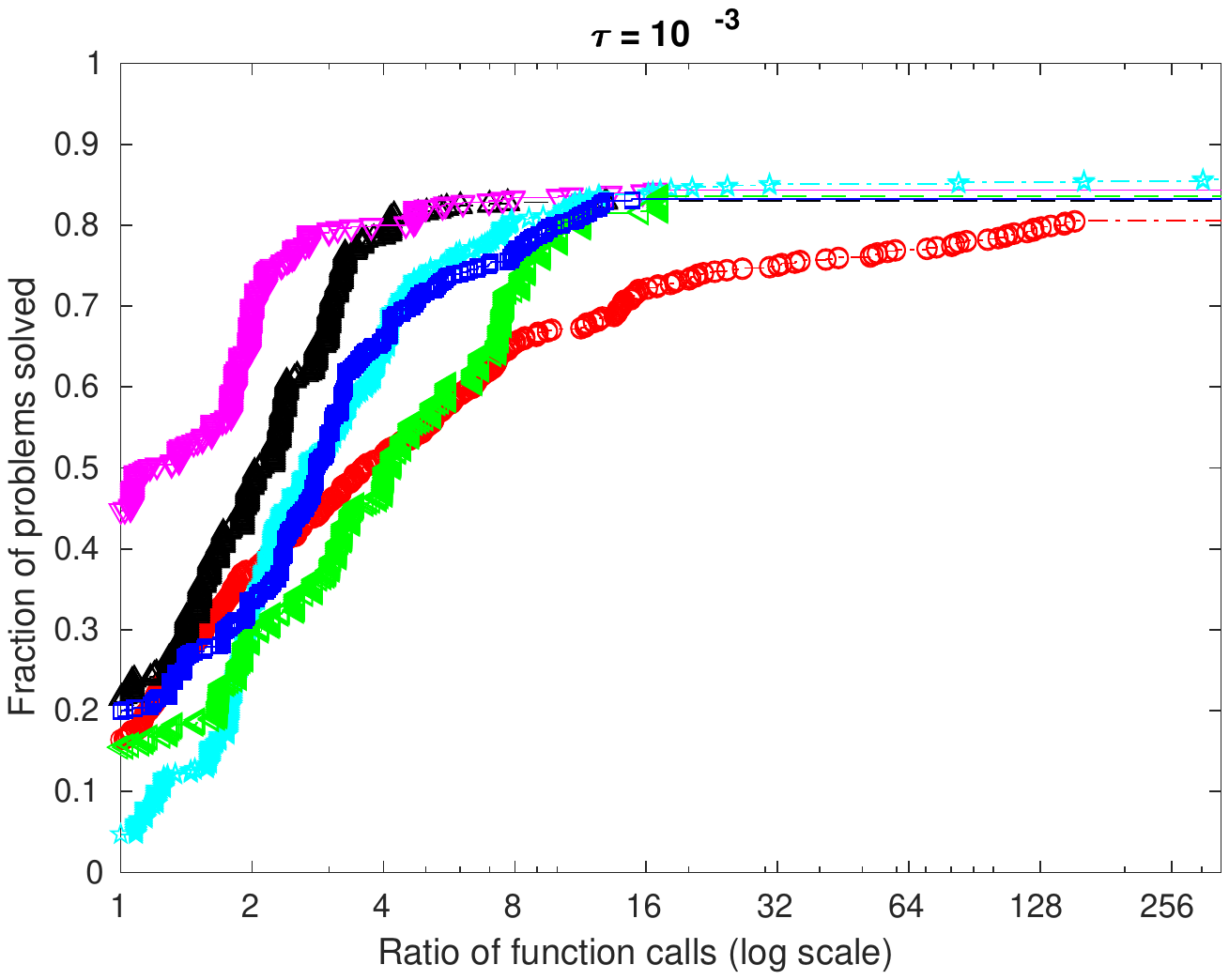}
\end{minipage}
\vskip -50ex
\hskip -20ex
\centering
\captionsetup{justification=centering}
\begin{minipage}[t]{0.47\textwidth}
  \includegraphics[scale=0.5]{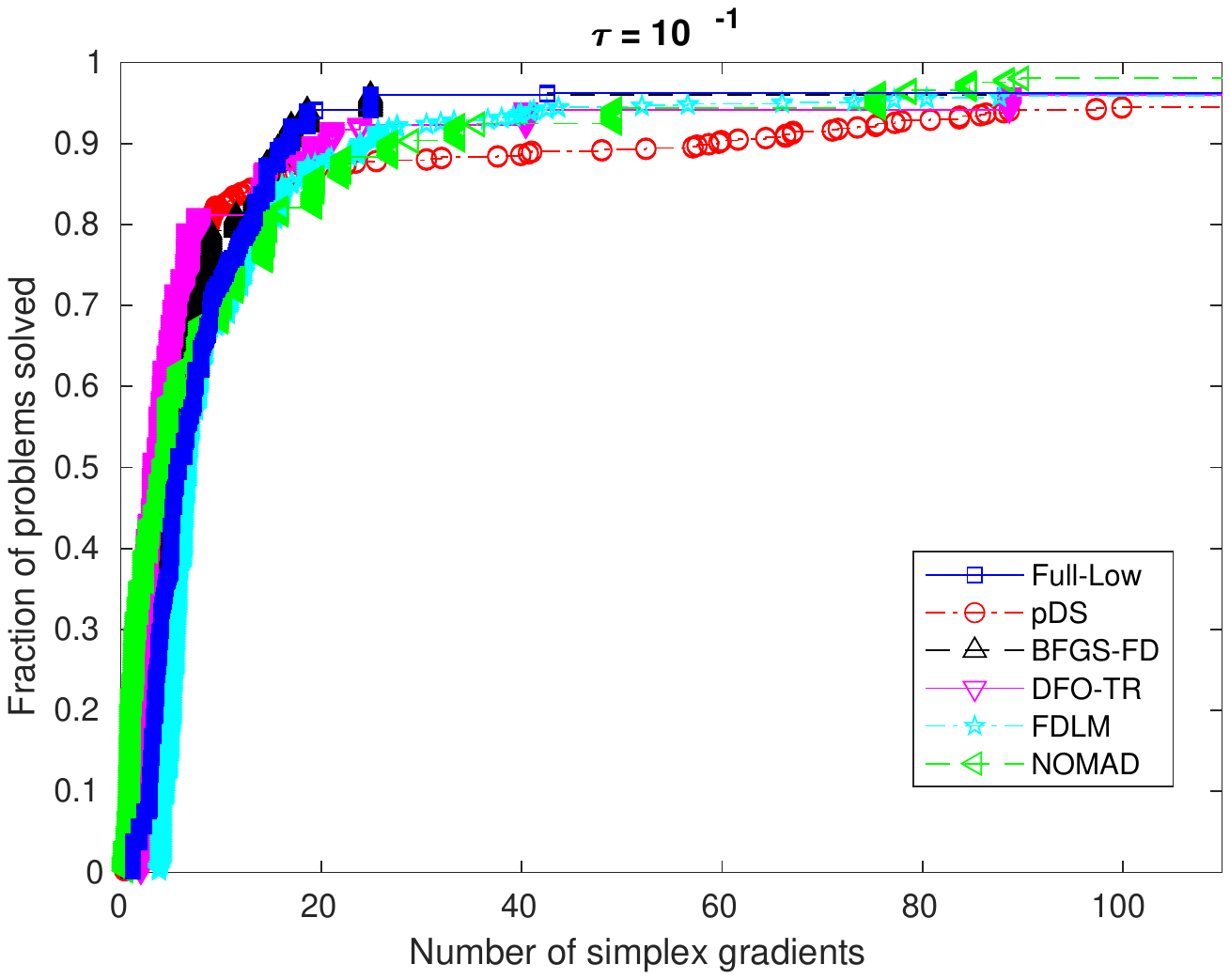}
\end{minipage}
 \hskip 5ex
\begin{minipage}[t]{0.47\textwidth}
  \includegraphics[scale=0.5]{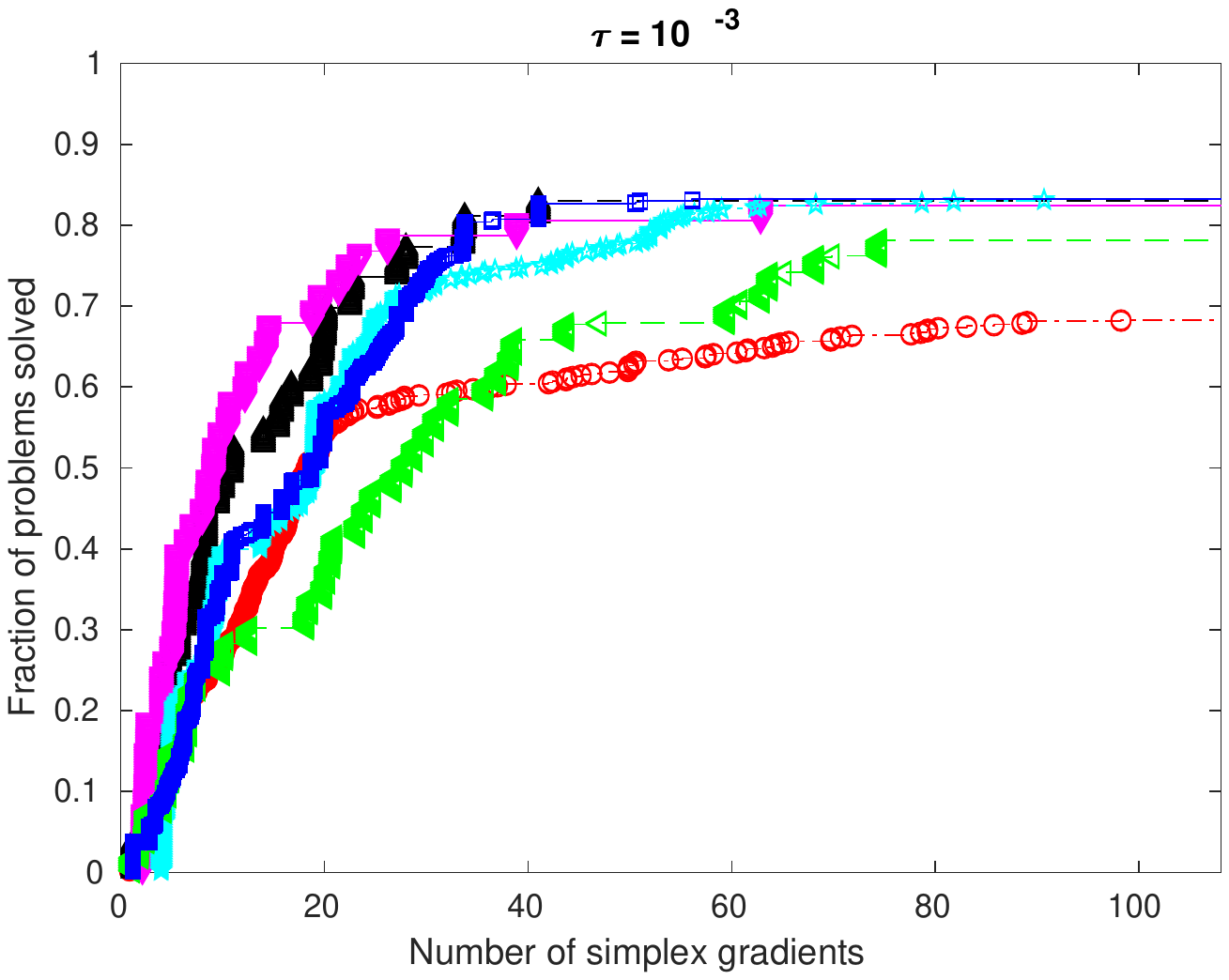}
\end{minipage}
\vskip -25ex
    \caption{Performance profiles (first row) and data profiles (second row) with $\tau = 10^{-1}, 10^{-3}$ of the 6 solvers.  Results for the additive deterministic noise problems (modified from~\cite{JJMore_SMWild_2009}).}
    \centering
\label{fig:comp4-deteradd}
\end{figure}

From Figure \ref{fig:comp4-deteradd}, one can see that for additive deterministic noise the best performance is the one by \DFOTR{}, which is not surprising. For these problems, it is still the case that
\FLE{} has the same efficiency and robustness as \texttt{BFGS-FD}, due to the fact that \texttt{BFGS-FD} (i.e., always doing \FullEval{} iterations) has a superior performance compared to \texttt{pDS} (i.e., always doing \LowEval{} iterations).
In the additive stochastic case (Figure~\ref{fig:comp4-noisyadd}), \NOMAD{} is clearly the best in terms of efficiency and robustness. Even though \texttt{BFGS-FD} is the worst, \FLE{} is still robust compared to the remaining methods due to the incorporation of \texttt{pDS} in the \LowEval{} iterations.

\begin{figure}[H]
\vskip -30ex
\hskip -20ex
\centering
\captionsetup{justification=centering}
\begin{minipage}[t]{0.47\textwidth}
  \includegraphics[scale=0.5]{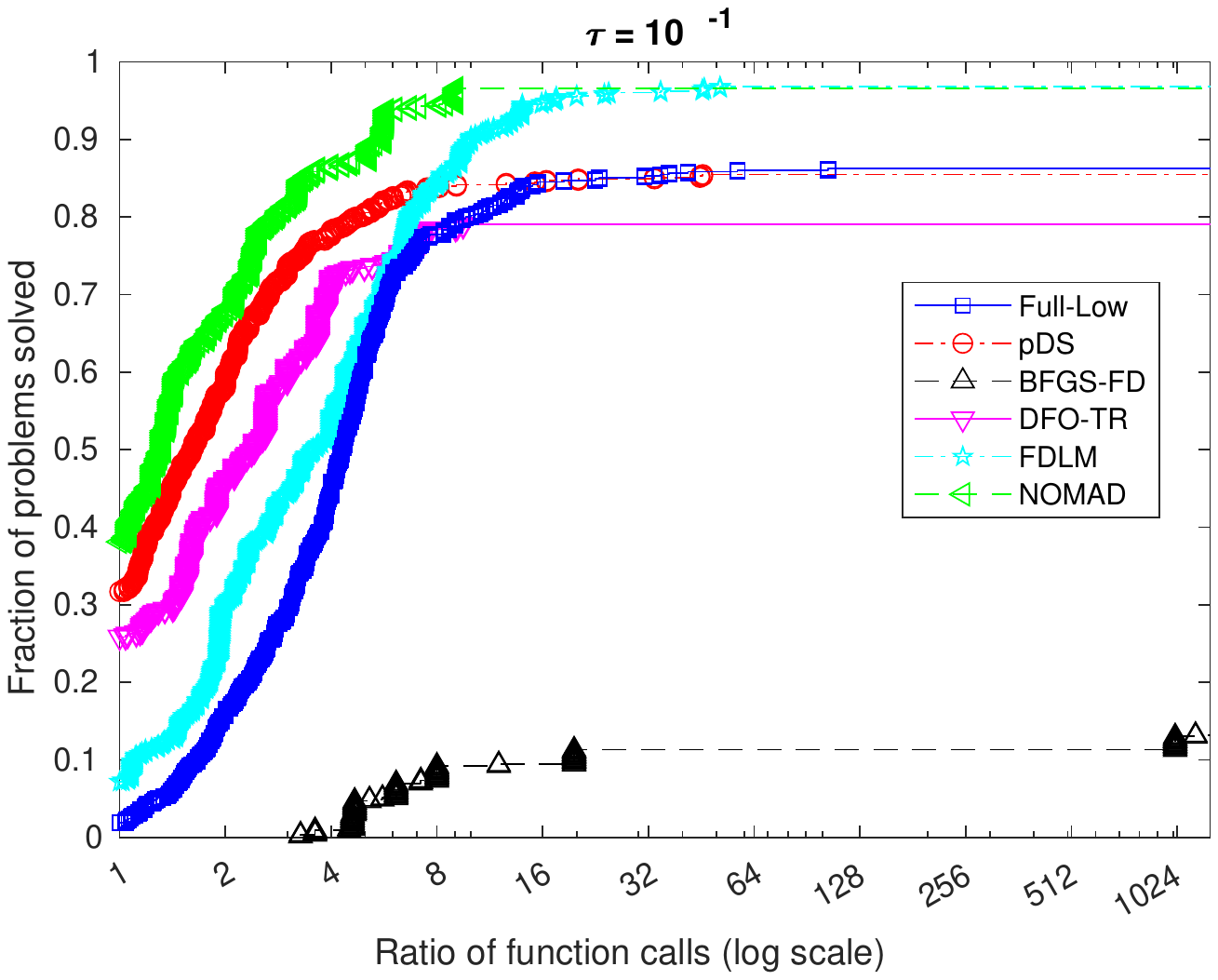}
\end{minipage}
 \hskip 5ex
\begin{minipage}[t]{0.47\textwidth}
  \includegraphics[scale=0.5]{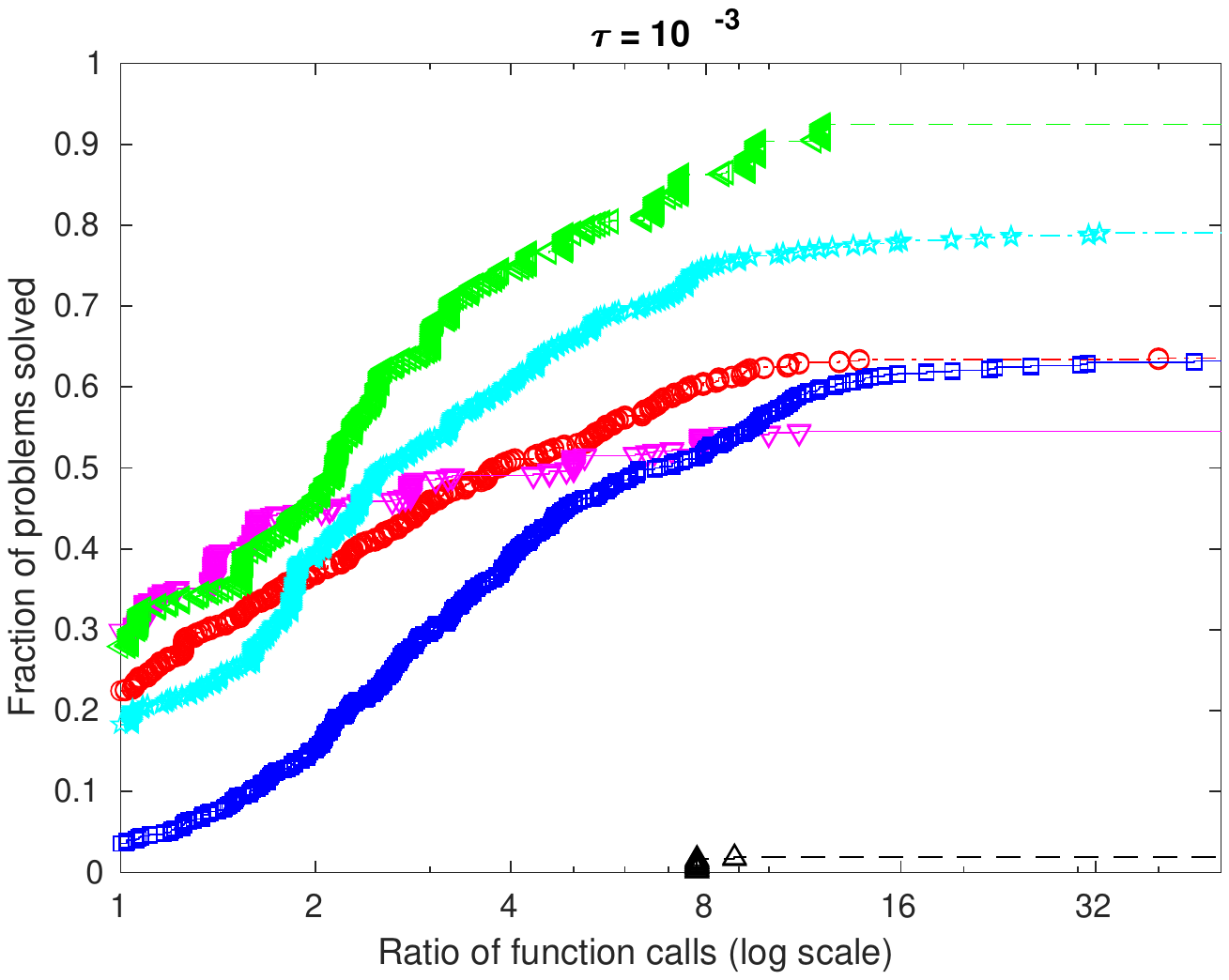}
\end{minipage}
\vskip -50ex
\hskip -20ex
\centering
\captionsetup{justification=centering}
\begin{minipage}[t]{0.47\textwidth}
  \includegraphics[scale=0.5]{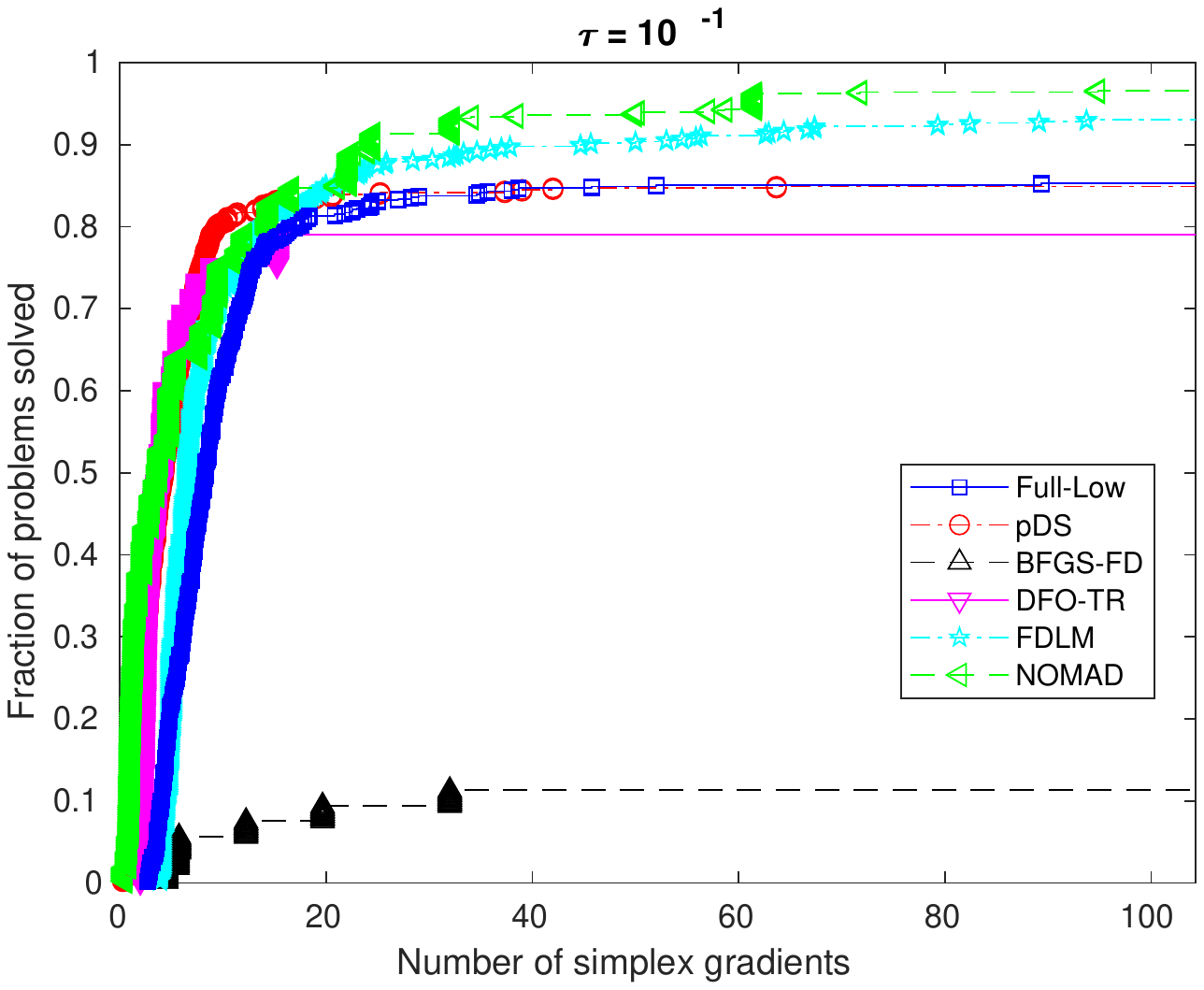}
\end{minipage}
 \hskip 5ex
\begin{minipage}[t]{0.47\textwidth}
  \includegraphics[scale=0.5]{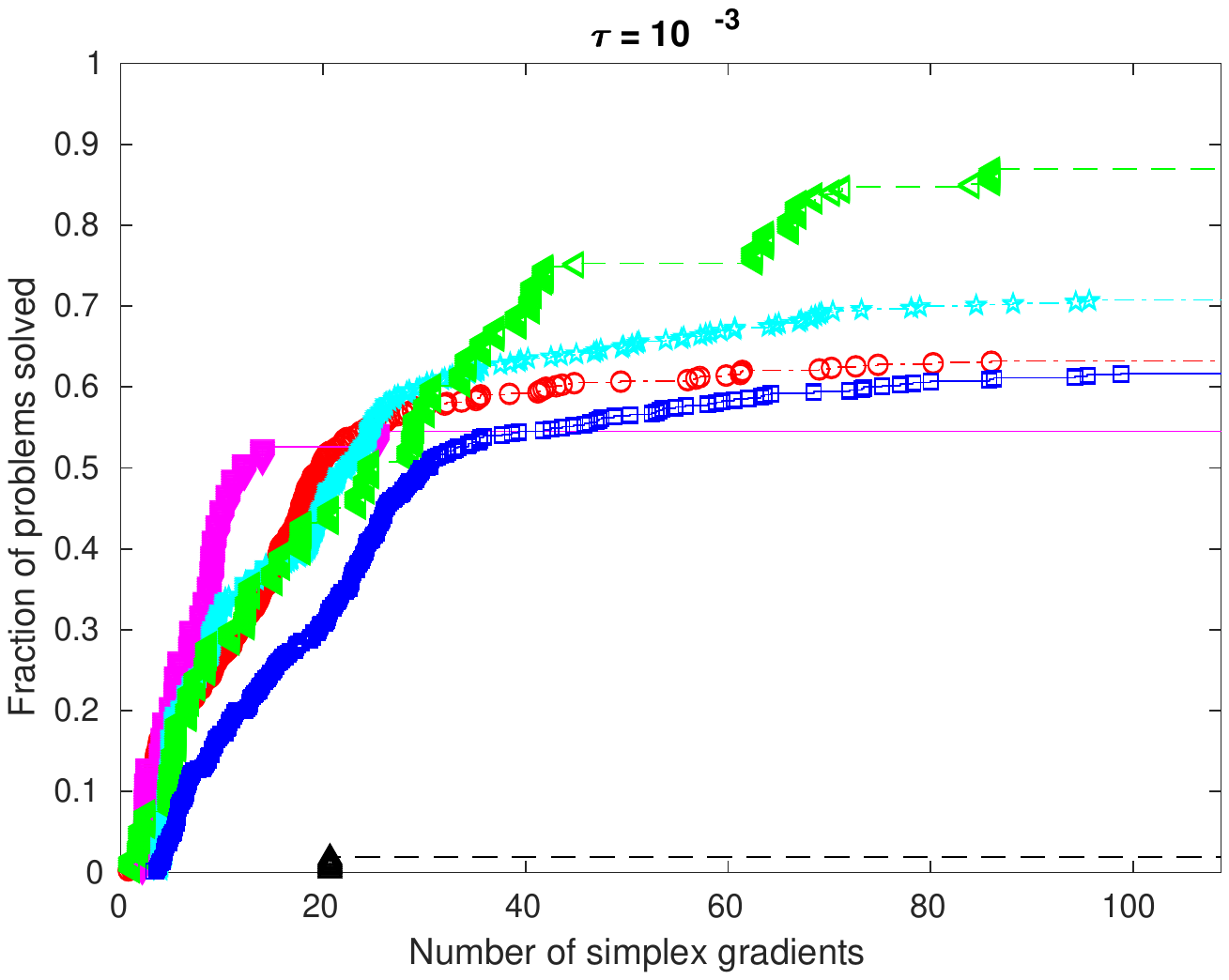}
\end{minipage}
\vskip -25ex
    \caption{Performance profiles (first row) and data profiles (second row) with $\tau = 10^{-1}, 10^{-3}$ of the 6 solvers.  Results for the additive stochastic noise problems (modified from~\cite{JJMore_SMWild_2009}).}
    \centering
\label{fig:comp4-noisyadd}
\end{figure}

In the multiplicative deterministic case (see Figure~\ref{fig:comp4-determult}), we can see that \NOMAD{} is robust regardless of the accuracy. \FLE{} is a little more efficient for high accuracy. \FDLM{} performs better than in the additive deterministic case. We can see that for high accuracy, \FLE{} performs better than just doing \FullEval{} or \LowEval{} iterations. In particular, by looking at data profiles, it solves a higher percentage of problems within small budgets.

\begin{figure}[H]
\vskip -30ex
\hskip -20ex
\centering
\captionsetup{justification=centering}
\begin{minipage}[t]{0.47\textwidth}
  \includegraphics[scale=0.5]{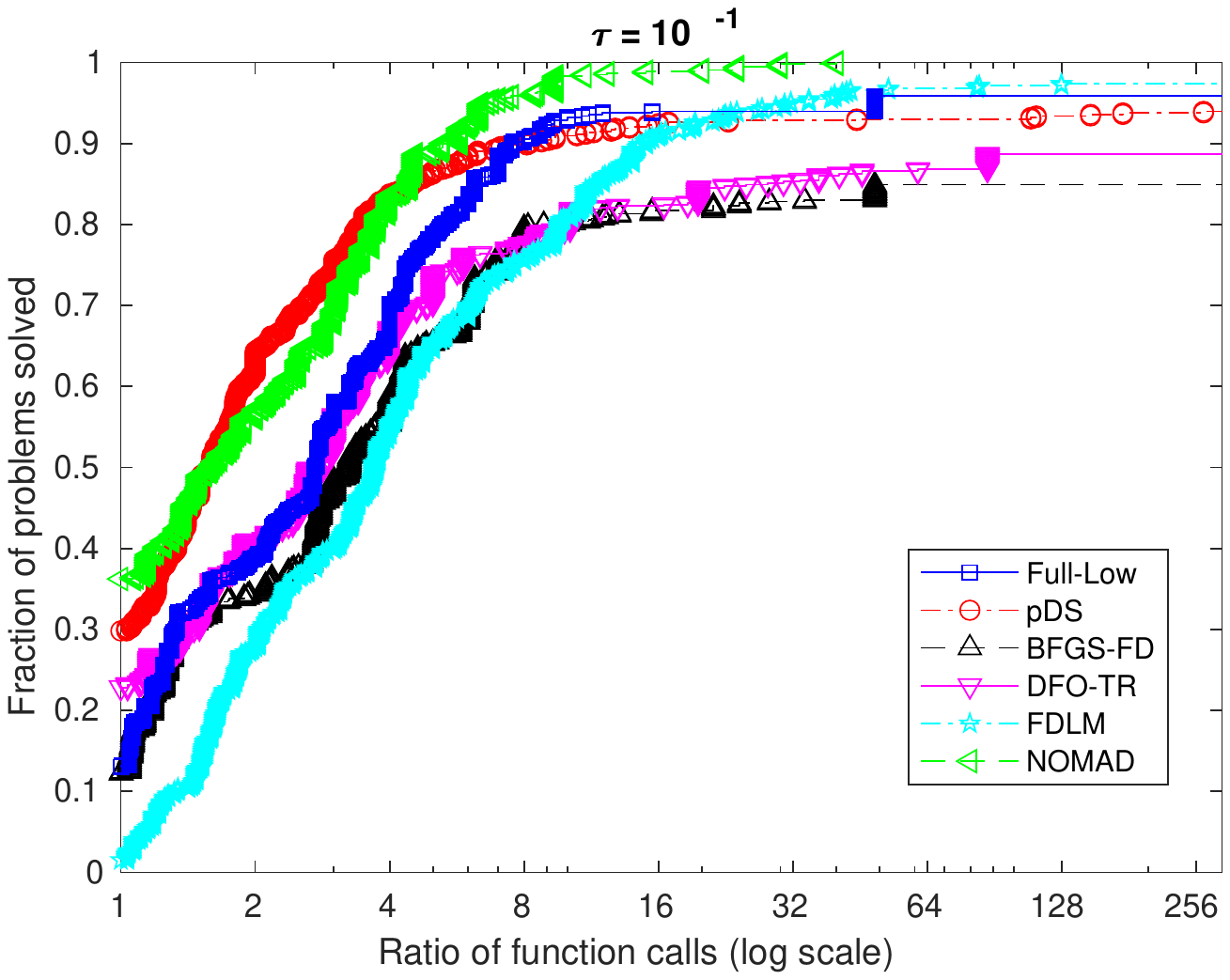}
\end{minipage}
 \hskip 5ex
\begin{minipage}[t]{0.47\textwidth}
  \includegraphics[scale=0.5]{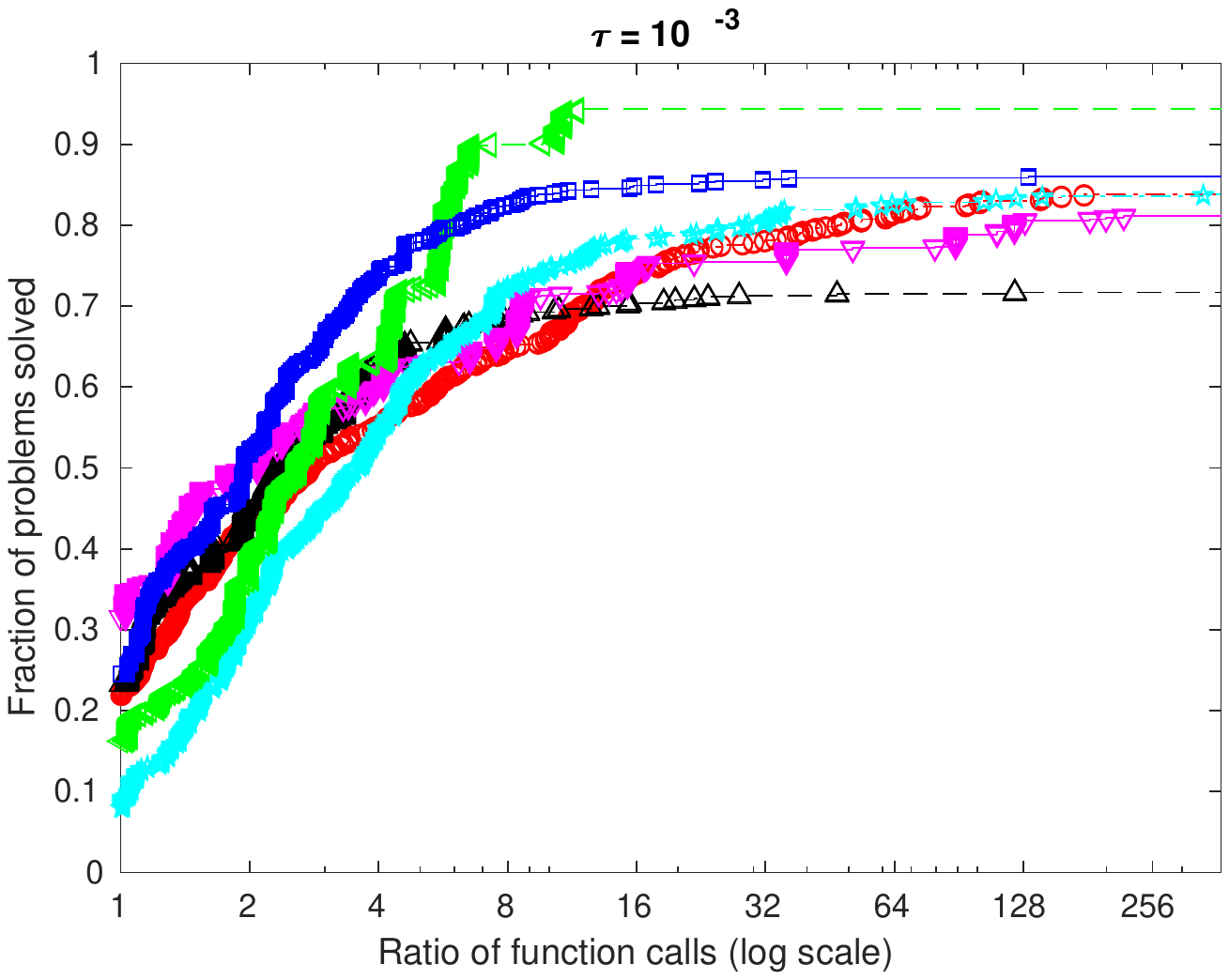}
\end{minipage}
\vskip -50ex
\hskip -20ex
\centering
\captionsetup{justification=centering}
\begin{minipage}[t]{0.47\textwidth}
  \includegraphics[scale=0.5]{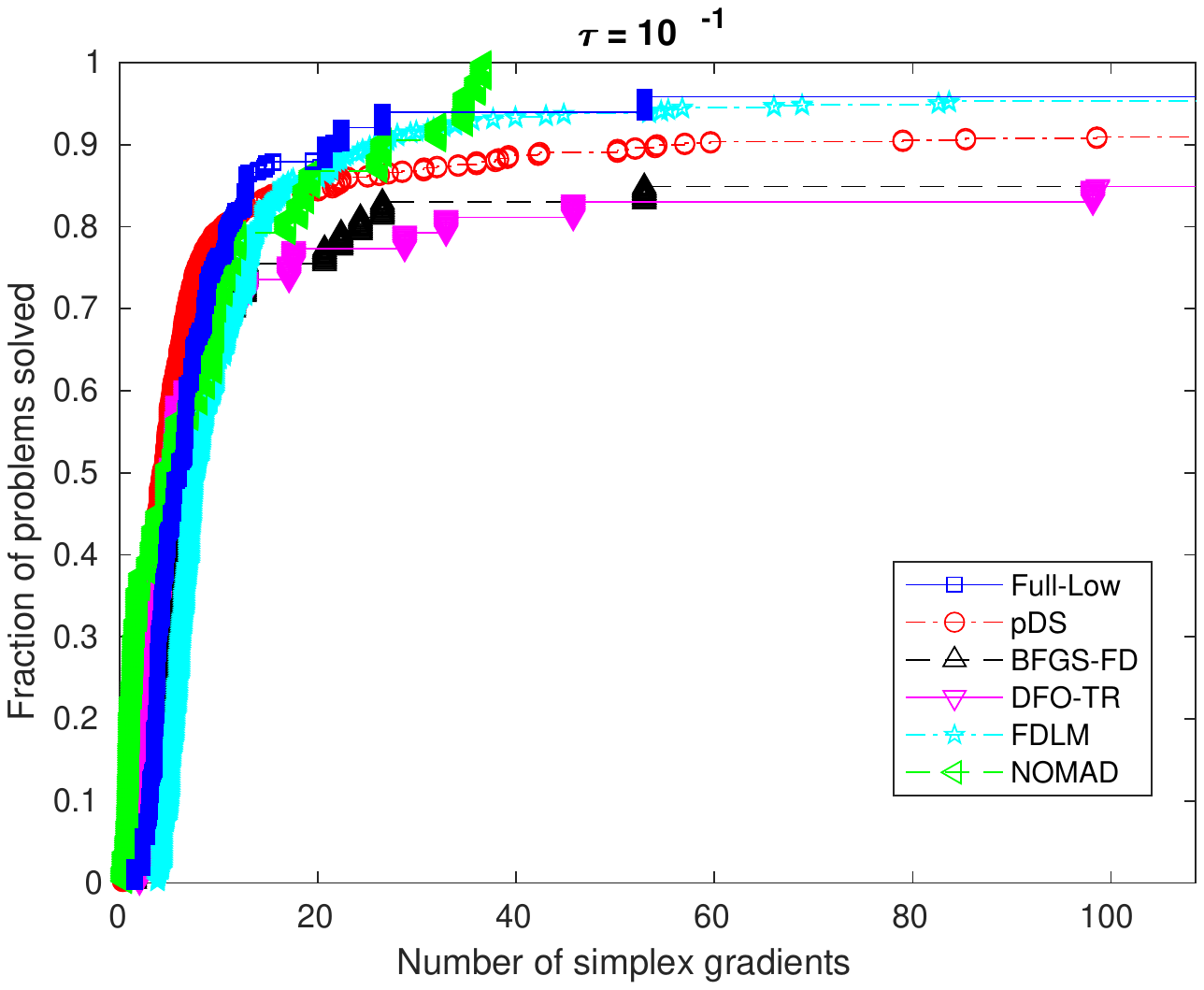}
\end{minipage}
 \hskip 5ex
\begin{minipage}[t]{0.47\textwidth}
  \includegraphics[scale=0.5]{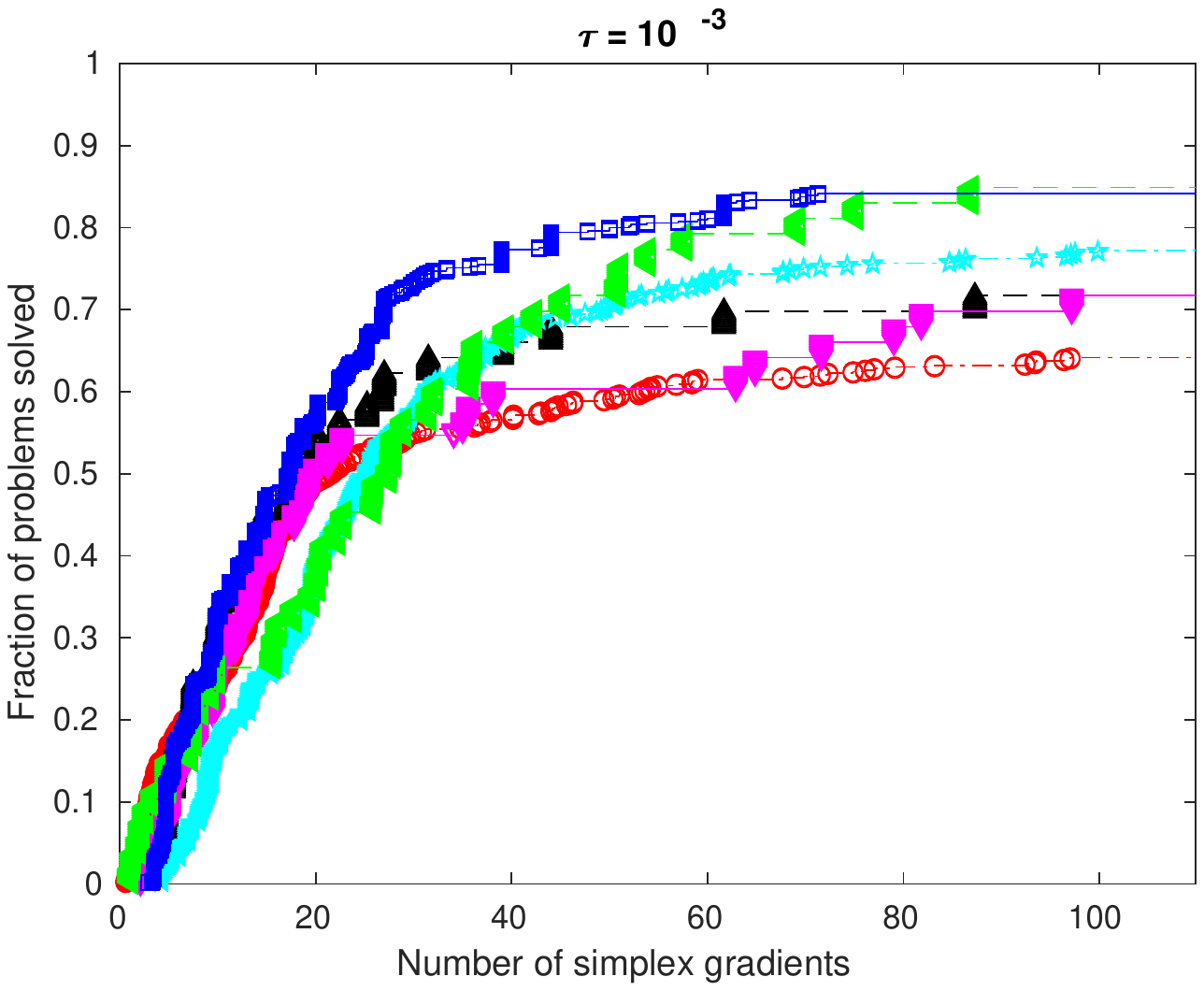}
\end{minipage}
\vskip -25ex
\caption{Performance profiles (first row) and data profiles (second row) with $\tau = 10^{-1}, 10^{-3}$ of the 6 solvers.  Results for the multiplicative deterministic noise problems in~\cite{JJMore_SMWild_2009}.}
\label{fig:comp4-determult}
\end{figure}

The results for the stochastic multiplicative noise are given in Figure~\ref{fig:comp4-noisymult}. The best solver is \NOMAD{} for both high and low accuracy. This time, the curve corresponding to \FLE{} is no longer above the \texttt{BFGS-FD} and \texttt{pDS}, instead it is in between them. This is because \texttt{BFGS-FD} performs poorly when~$h$ is equal to the square root of machine precision (and the noise is ``differentiated''). Note that \FDLM{} exhibits the second best robustness for low accuracy even though it relies mostly on BFGS.

\begin{figure}
\vskip -30ex
\hskip -25ex
\centering
\captionsetup{justification=centering}
\begin{minipage}[t]{0.47\textwidth}
  \includegraphics[scale=0.5]{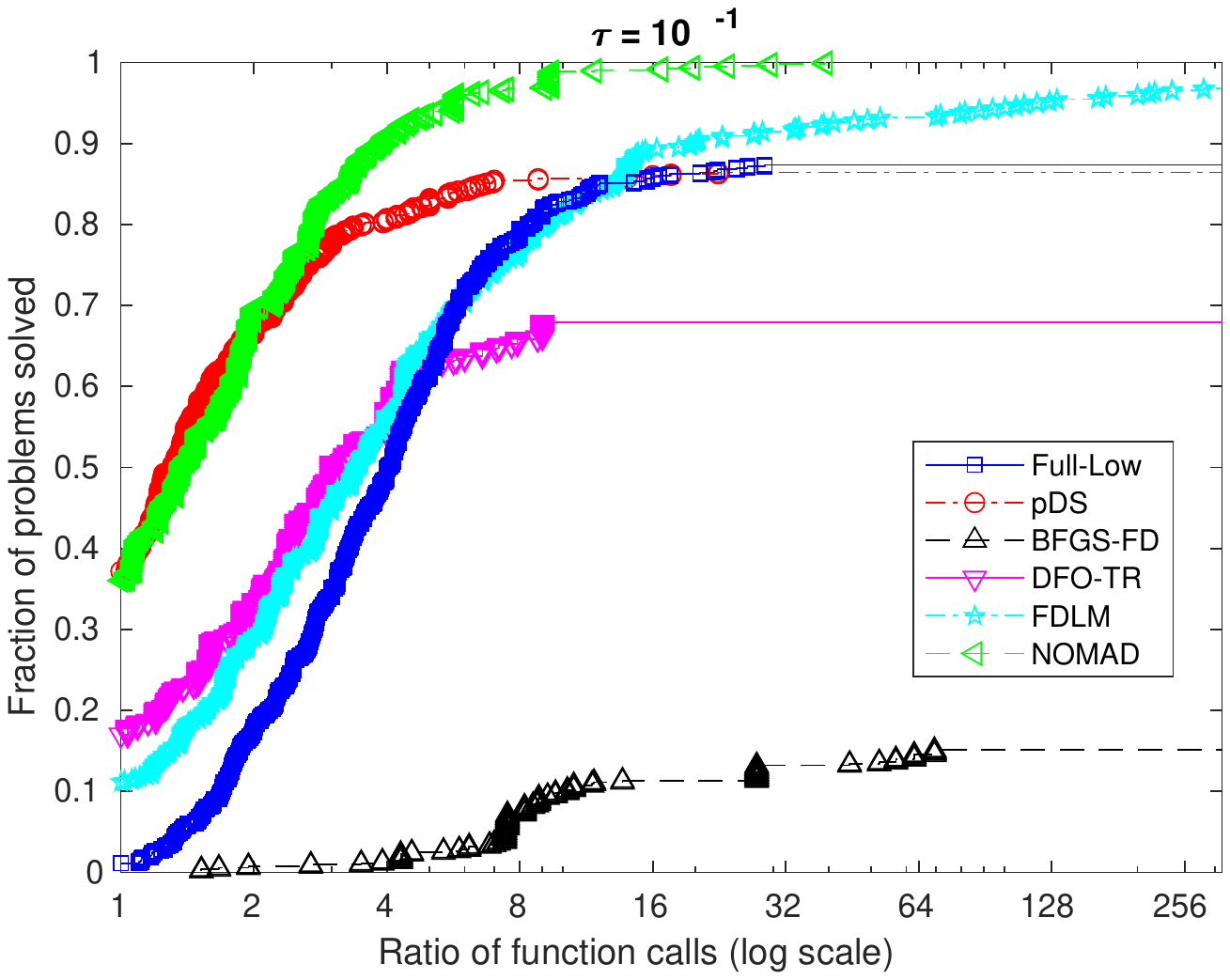}
\end{minipage}
 \hskip 5ex
\begin{minipage}[t]{0.47\textwidth}
  \includegraphics[scale=0.5]{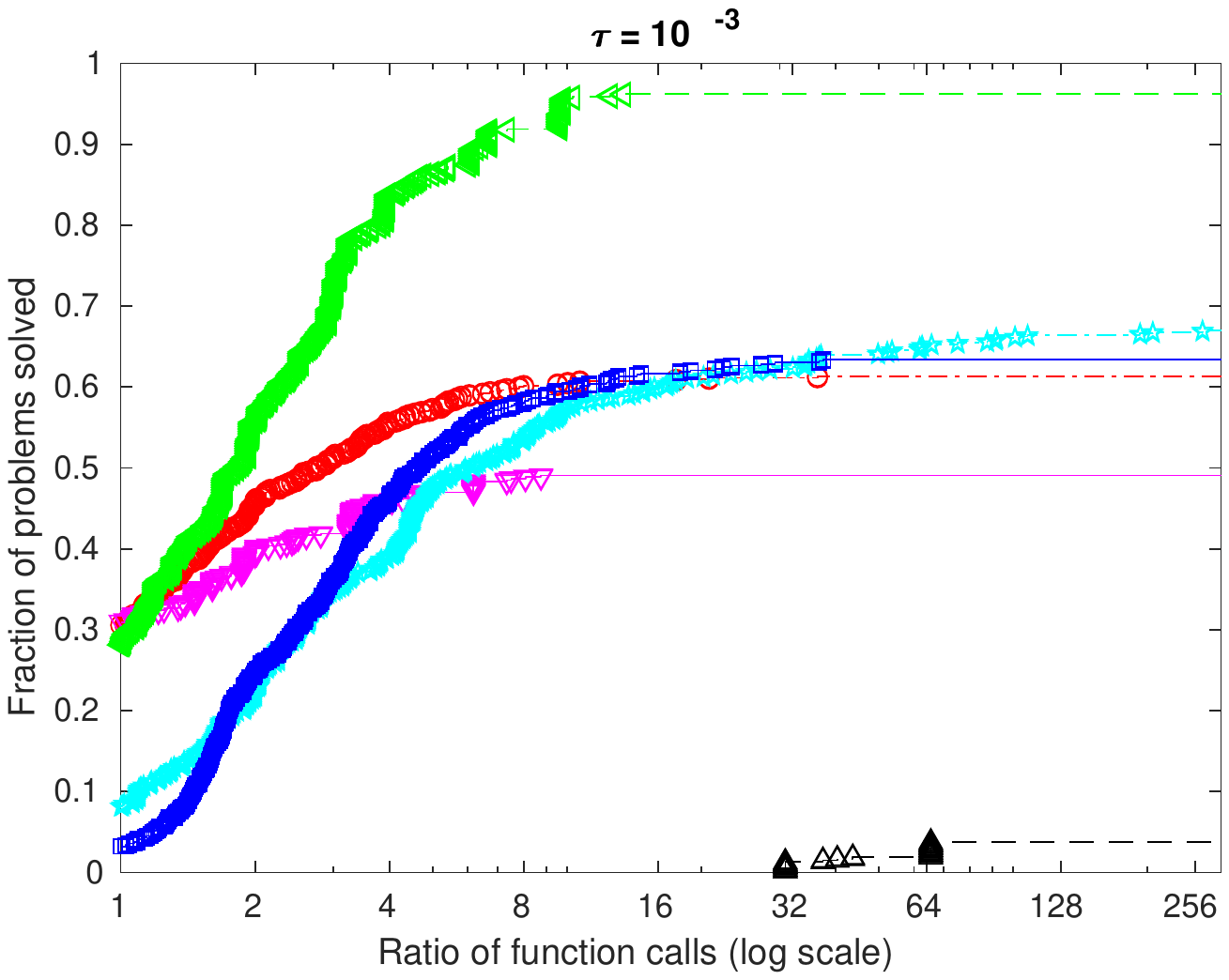}
\end{minipage}

\vskip -50ex
\hskip -25ex
\centering
\captionsetup{justification=centering}
\begin{minipage}[t]{0.47\textwidth}
  \includegraphics[scale=0.5]{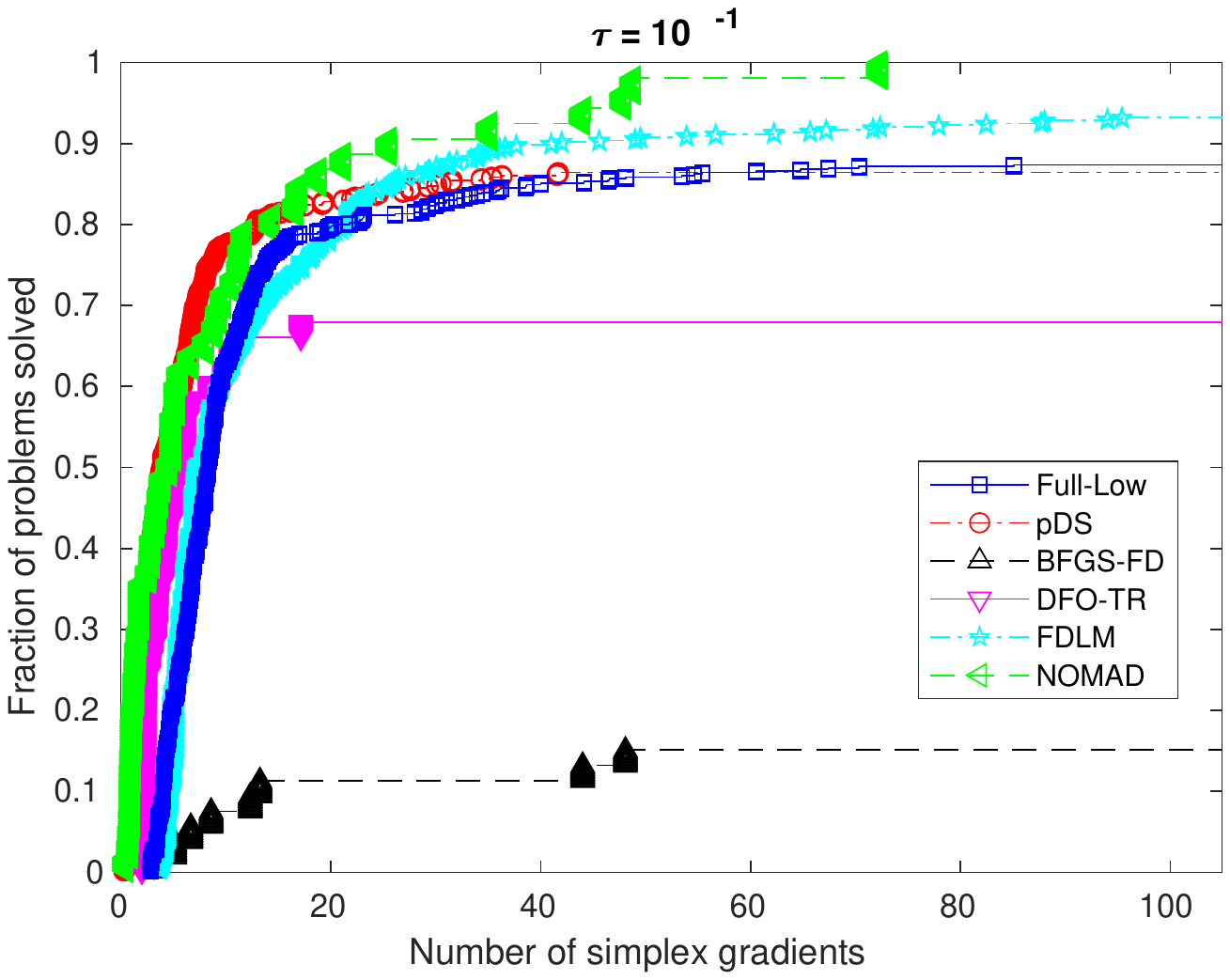}
\end{minipage}
 \hskip 5ex
\begin{minipage}[t]{0.47\textwidth}
  \includegraphics[scale=0.5]{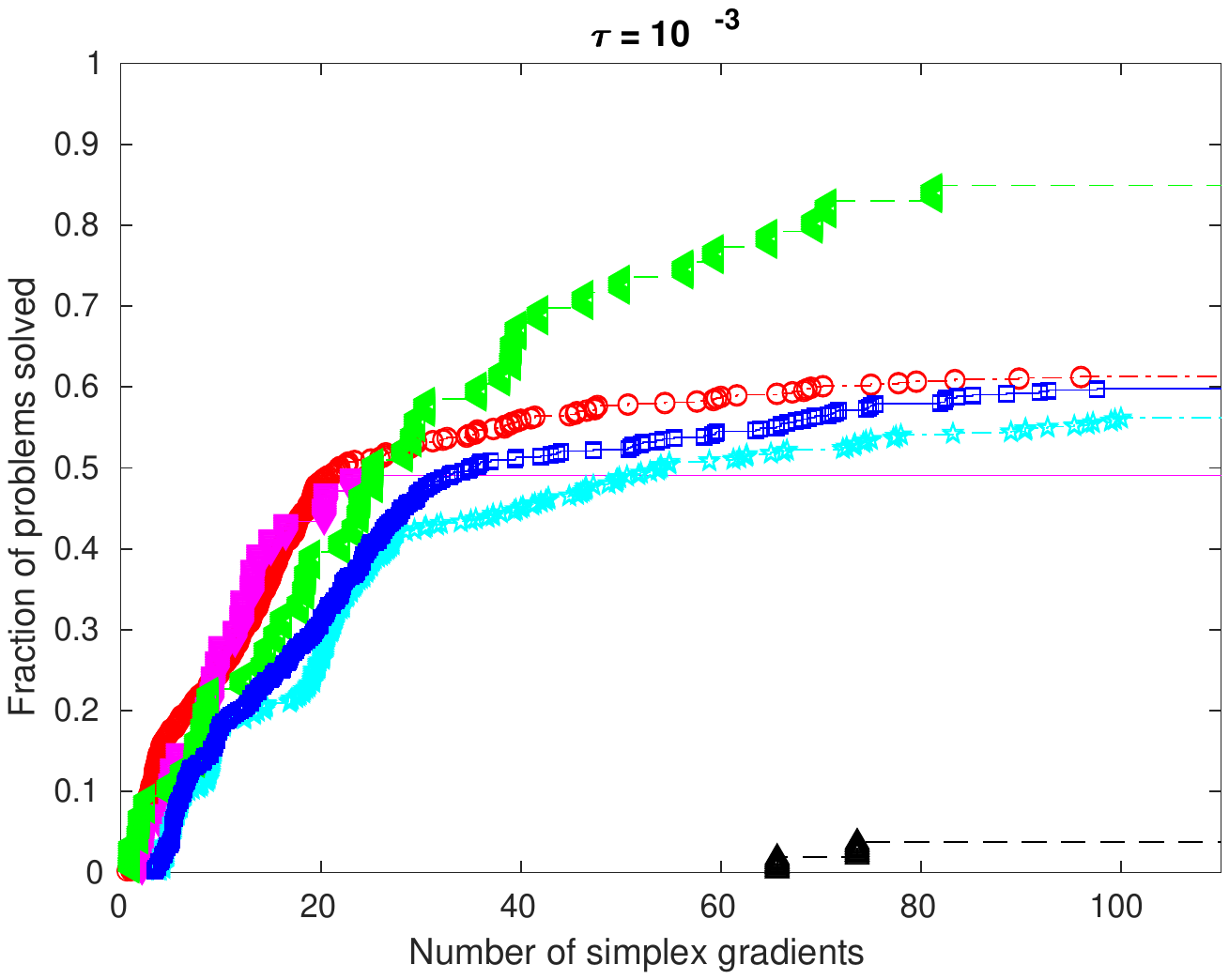}
\end{minipage}
\vskip -25ex
\caption{\centering Performance profiles (first row) and data profiles (second row) with $\tau = 10^{-1}$, $10^{-3}$ for the 6 solvers. Results for the multiplicative stochastic noise problems in~\cite{JJMore_SMWild_2009}.}

\label{fig:comp4-noisymult}
\end{figure}

Taking into consideration all problem types, $\FLE{}$ stands out as the best overall performer, when both efficiency and robustness are considered. For the noisy problems, we have also tried noise levels $\epsilon_f=10^{-2},10^{-4}$, but the relative positions of the curves in the profiles remain the same.

\section{Concluding remarks}
\label{sec:conclusions}

We introduced a new framework for unconstrained derivative-free optimization, consisting of the rigorous integration of two different methodologies. The goal was to combine the strengths of both methodologies in order to achieve solid numerical behavior (efficiency and robustness) regardless of the smoothness or noise regime of the objective function. The first methodology is related to the computation of {\it fully} linear models (in our case by computing finite difference gradients). The second is based on the {\it low} evaluation paradigm of probabilistic direct search. Our convergence analysis is novel in the way it connects the two methodologies, to rigorously extract their best properties.


The \FLE{} framework can be analyzed and implemented with other choices for the \FullEval{} and \LowEval{} iterations. For example, a trust-region step based on fully linear models \cite{ASBandeira_KScheinberg_LNVicente_2012} is a candidate for the \FullEval{} iterations. Other possibilities include line search methods that use simplex gradients \cite{CTKelley_2011} or deterministic direct search (with complete polling on a set of points defined by a PSS) \cite{TGKolda_RMLewis_VTorczon_2003}, instead of FD. Candidates for \LowEval{} include, for instance, randomized Gaussian smoothing methods (one-point~\cite{YNesterov_2011} or two-point approaches~\cite{JCDuchi_etal_2015}).


There are a number of future research items to be investigated, and two stand out naturally. The extension to the constrained case, in particular to bound and linear constraints, and the consideration of DFO problems with larger dimensions. We plan to report on these topics in future manuscripts.


\small

\bibliographystyle{plain}
\bibliography{ref-full-low-eval}

\end{document}